\documentclass[11pt]{article}
\usepackage{amssymb,amsmath,amsthm,rotating,setspace}
\usepackage{enumerate}
\usepackage{cases}
\numberwithin{equation}{section}

\newcommand{\n}{\noindent}
\newcommand{\bb}[1]{\mathbb{#1}}

\newcommand{\vp}{\varepsilon}

\theoremstyle{plain}
\newtheorem{thm}{Theorem}[section]
\newtheorem{lem}{Lemma}[section]

\newtheorem{cor}{Corollary}[section]

\theoremstyle{definition}

\newtheorem{rem}{Remark}

\begin{document}

\title{Nonexistence of Positive Supersolutions of Nonlinear Biharmonic
  Equations without the Maximum Principle}

\author{Marius Ghergu\footnote{School of Mathematical Sciences,
    University College Dublin, Belfield, Dublin 4, Ireland; {\tt
      marius.ghergu@ucd.ie}} and Steven
  D.~Taliaferro\footnote{Mathematics Department, Texas A\&M
    University, College Station, TX 77843-3368; {\tt
      stalia@math.tamu.edu}} 
\footnote{Corresponding author, Phone 001-979-845-3261, Fax 001-979-845-6028}}

\date{}
\maketitle


\begin{abstract}
  We study classical positive solutions of the biharmonic inequality
\begin{equation}\label{abs}
 -\Delta^2 v \geq f(v)
\end{equation}
in exterior domains in $\mathbb{R}^n$ where $f:(0,\infty)\to
(0,\infty)$ is continuous function.  We give lower bounds on the
growth of $f(s)$ at $s=0$ and/or $s=\infty$ such that inequality
\eqref{abs} has no $C^4$ positive solution in any exterior domain of
$\mathbb R^n$. Similar results were obtained by Armstrong and Sirakov
[{\it Nonexistence of positive supersolutions of elliptic equations
  via the maximum principle,} Comm. Partial Differential Equations 36
(2011) 2011-2047] for $-\Delta v\ge f(v)$ using a method which depends
only on properties related to the maximum principle. Since the maximum
principle does not hold for the biharmonic operator, we adopt a
different approach which relies on a new representation formula and an
a priori pointwise bound for nonnegative solutions of $-\Delta^2u \ge
0$ in a punctured neighborhood of the origin in $\mathbb{R}^n$.
\end{abstract}

\section{Introduction}\label{sec1}
Using a method which depends only on properties related to the maximum
principle, Armstrong and Sirakov \cite{AS} proved the following two
nonexistence result for positive solutions of
\begin{equation}\label{eq1.1}
 -\Delta v\geq f(v)
\end{equation}
in exterior domains in $\mathbb{R}^n$.

\begin{thm}[Armstrong and Sirakov \cite{AS}]\label{thm1.1.1}
Assume that $n\geq 3$ and the nonlinearity $f:(0,\infty)\to
(0,\infty)$ is continuous and satisfies
 \begin{equation}\label{eq1.2}
  \liminf_{s\to 0^{+}} \frac{f(s)}{s^{1 + \frac{2}{n-2}}} > 0.
 \end{equation}
Then the inequality \eqref{eq1.1} has no positive solution in any exterior
domain of $\mathbb{R}^n$.
\end{thm}

The exponent $1+\frac{2}{n-2}$ in \eqref{eq1.2} is optimal because,
as pointed out in \cite{AS}, for each constant $\lambda >
1+\frac{2}{n-2}$ there exists a positive constant $C$ such that a
solution of $-\Delta v=v^\lambda$ in $\mathbb{R}^n \setminus \{0\}$,
which tends to zero as $|y|\to \infty$,
is $v(y)=C|y|^{\frac{-2}{\lambda -1}}$.

\begin{thm}[Armstrong and Sirakov \cite{AS}]\label{thm1.1.2}
Let $f$ be a positive continuous function on $(0,\infty)$ which
satisfies
\begin{equation}\label{eq1.2.1} 
\lim_{s\to\infty}e^{as}f(s)=\infty\quad\text{for every $a>0$.}
\end{equation}
Then the inequality \eqref{eq1.1} has no positive solution in any
exterior domain of $\mathbb{R}^2$. 
\end{thm}

Theorem \ref{thm1.1.2} is also sharp as explained in \cite{AS}.

In this paper we study the nonexistence of positive
solutions of the biharmonic inequality
\begin{equation}\label{eq1.1.3}
 -\Delta^2 v \geq f(v)
\end{equation}
in exterior domains in $\mathbb{R}^n$. When $n\ge 3$ we obtain the
following result.

\begin{thm}\label{thm1.1.3}
Let $f$ be a positive continuous function on $(0,\infty)$ which
satisfies
\begin{equation}\label{eq1.4.1}
  \liminf_{s \to 0^{+}} \frac{f(s)}{s^{1+\frac{4}{n-2}}} >0 \quad\text{and}\quad
   \lim_{s\to \infty} \frac{f(s)}{s^{-1}} = \infty.
 \end{equation} 
Then the inequality \eqref{eq1.1.3} has no $C^4$ positive solution in any
exterior domain of $\mathbb{R}^n$, $n\ge 3$. 
\end{thm}

\begin{rem}\label{rem1}
  The exponent $1+\frac{4}{n-2}$ in \eqref{eq1.4.1} is optimal because
  for each constant $\lambda \in (1+\frac{4}{n-2}, 1+\frac{4}{n-4})$
  (resp. $\lambda>1+\frac{4}{n-2}$) there exists a positive constant
  $C$ such that a solution of
 \begin{equation}\label{eq1.5}
  -\Delta^2 v=v^\lambda \quad \text{in} \quad 
\mathbb{R}^n \setminus \{0\},\ \text{$n\ge 5$ (resp. $n=3$ or 4)}, 
 \end{equation}
which tends to zero as $|y|\to \infty$, 
is $v(y)=C|y|^{\frac{-4}{\lambda - 1}}$.
\end{rem}

\begin{rem}\label{rem2}
  The exponent $-1$ in \eqref{eq1.4.1} is optimal because for each
  constant $\lambda<-1$, (resp. $\lambda\in(-3,-1)$), there exists a
  positive constant $C$ such that a solution of \eqref{eq1.5}, which
  tends to infinity as $|y|\to\infty$, is $v(y) =
  C|y|^{\frac{-4}{\lambda-1}}$.  
\end{rem}

\begin{rem}
  We conjecture that Theorem \ref{thm1.1.3} is true when in
  \eqref{eq1.4.1} the condition on $f$ at $\infty$ is replaced with
\begin{equation}\label{eq1.5.1}
\liminf_{s\to\infty} \frac{f(s)}{s^{-1}}>0.
\end{equation}
It can be shown that this conjecture is true under the added
assumption in Theorem \ref{thm1.1.3} that $v$ is radial.
\end{rem}

By Remarks \ref{rem1} and \ref{rem2}, we see, in strong
  contrast to Theorem \ref{thm1.1.1}, that a growth condition on $f$ at
  both $s=0$ and $s=\infty$ is necessary for nonexistence of positive
  solutions of \eqref{eq1.1.3} in exterior domains of $\mathbb R^n$,
  $n\ge 3$.

Our two dimensional result for \eqref{eq1.1.3} is the following
theorem. As is Theorem \ref{thm1.1.2} and in contrast to Theorem
\ref{thm1.1.3}, a growth condition on $f$ is only needed at $s=\infty$
for nonexistence of positive solutions of \eqref{eq1.1.3} in exterior
domains of $\mathbb R^2$

\begin{thm}\label{thm1.1.4}
Let $f$ be a positive continuous function on $(0,\infty)$ which
satisfies
\begin{equation}\label{eq1.4.2}
   \liminf_{s\to \infty} \frac{f(s)}{s^{-1}\log s} > 0.
 \end{equation} 
Then the inequality \eqref{eq1.1.3} has no $C^4$ positive solution in any
exterior domain of $\mathbb{R}^2$. 
\end{thm}

\begin{rem}\label{rem4}
  The exponent $-1$ in \eqref{eq1.4.2} is optimal because for each
  constant $\lambda\in (-2,-1)$ there exists a positive constant $C$
  such that a positive solution of 
\[
-\Delta^2 v\ge v^\lambda \quad\text{in}\quad 
\mathbb{R}^2\setminus B_e(0),
\] 
which tends to infinity as $|y|\to\infty$,
is $v(y)=2C|y|^2\log|y|-|y|^b$ where
$b=2\lambda +4\in(0,2)$.
\end{rem}

\begin{rem}
  We conjecture that Theorem \ref{thm1.1.4} is true when 
  \eqref{eq1.4.2} is replaced with \eqref{eq1.5.1}. 
By Lemma \ref{lem2.8} this conjecture is true under the added
assumption in Theorem \ref{thm1.1.4} that $v$ is radial.
\end{rem}

  Since there are continuous functions $f:(0,\infty)\to(0,\infty)$
  satisfying \eqref{eq1.4.1} (resp. \eqref{eq1.4.2}) which are not
  bounded below by a convex function $g:(0,\infty)\to(0,\infty)$
  satisfying \eqref{eq1.4.1} (resp. \eqref{eq1.4.2}), one cannot
  immediately reduce the the proof of Theorem \ref{thm1.1.3}
  (resp. Theorem \ref{thm1.1.4}) to an ODE problem by the standard
  method of averaging which consists of replacing $f$ in
  \eqref{eq1.1.3} with such a $g$, averaging the resulting inequality,
  and using Jensen's inequality. In particular, obtaining nonexistence
  results for \eqref{eq1.1.3} under assumption \eqref{eq1.4.1}
  (resp. \eqref{eq1.4.2}) is much more difficult than obtaining them,
  say, for
\begin{equation}\label{convex}
-\Delta^2 v\ge v^\lambda, 
\end{equation}
where $\lambda\in\mathbb R\setminus (0,1)$ is a constant, because the
function $f(v)=v^\lambda$ is convex. Our results when applied to
\eqref{convex} give the following corollary.

\begin{cor} \label{cor1} Suppose $\lambda\in \mathbb R$, $r_0>0$ and
  $n=2$ (resp. $n\ge 3$). Then \eqref{convex} has $C^4$ positive
  solutions in $\mathbb R^n\setminus B_{r_0}(0)$ if and only if
\[
\lambda<-1 \quad
(\text{resp. $\lambda<-1 \quad\mbox{or}\quad \lambda> 1+\frac{4}{n-2}$}).
\] 
\end{cor}

\begin{proof}
  The ``if'' part of the corollary follows by scaling, if necessary,
  the examples in Remarks \ref{rem1}, \ref{rem2}, and \ref{rem4}. The
  ``only if'' part of the corollary follows from Theorems
  \ref{thm1.1.3} and \ref{thm1.1.4} when $\lambda\not= -1$ and from
  Lemmas \ref{lem2.8} and \ref{lem2.10} when $\lambda=-1$.
\end{proof}

The result in Corollary \ref{cor1} above is different from the study
of \eqref{convex} in the whole space $\mathbb R^n$. Mitidieri and
Pohozaev \cite[Theorem 7.1, pg. 31]{MP} proved that \eqref{convex} has
no solution in $\mathbb R^n$ if $1<\lambda\leq 1+4/(n-4)$. In
particular, no entire solution exists for all $\lambda>1$ in
dimensions $n=3$ and $n=4$.

Let us briefly describe the methods we employ in this paper to deal
with the biharmonic inequality \eqref{eq1.1.3}.  The method used in
\cite{AS} to prove Theorems \ref{thm1.1.1} and \ref{thm1.1.2} depends
only on properties related to the maximum principle.  Since the
maximum principle does not hold for the biharmonic operator, we adopt
a different approach to prove Theorems \ref{thm1.1.3} and
\ref{thm1.1.4} which relies on a new representation formula and an a
priori pointwise bound for nonnegative solutions of $-\Delta^2u \ge 0$
in a punctured neighborhood of the origin in $\mathbb{R}^n$,  which we
state in Appendix \ref{secA}.  We assume for contradiction that there
exists a positive solution $v(y)$ of \eqref{eq1.1.3} in an exterior
domain and apply this representation formula \eqref{eq3.8} and
pointwise bound \eqref{eq3.6} to the $2$-Kelvin transform $u(x)$ of
the function $v(y)$. A crucial step in our approach is to show using
\eqref{eq3.8} that the estimate \eqref{eq3.7} can be improved to
$$
\int_{|x|<1} -\Delta^2 u(x)\,dx<\infty.
$$
This will then imply that
$$
\int_{|x|<r}u(x)\,dx=o(r^3) \quad\mbox{ as }r\to 0^+,
$$
which will allow us to obtain with the help of Lemma \ref{lem2.5} a
refined representation formula for $u$, the crucial term of which is,
instead of \eqref{eq3.9}, 
\[
\hat N(x)=\int_{B_1(0)}\Phi(x-y)\Delta^2u(y)\,dy.
\]
Here $\Phi$ is the fundamental solution of $\Delta^2$ in ${\bb R}^n$ 
given by
\begin{numcases}{\Phi(x):=A}
|x|^{4-n} & if $n\ge 5$ \label{eq1.10}\\
\log\frac{e}{|x|} & if $n=4$ \label{eq1.11}\\
-|x| & if $n=3$ \label{eq1.12}\\
-|x|^2\log\frac{e}{|x|} & if $n=2$\label{eq1.13}
\end{numcases}
where $A=A(n)$ is a \emph{positive} constant. 
Finally we are able to raise a contradiction by providing with the help
of Lemma \ref{lem2.1} various estimates as $r\to 0^+$ of expressions
involving $\int_{|x|=r}\hat N(x)\, dx$.

The form and sign of the fundamental solution $\Phi$ have a large
influence on the proofs of Theorems \ref{thm1.1.3} and \ref{thm1.1.4}. The
proofs in cases \eqref{eq1.10} and \eqref{eq1.11} are similar but very
different from the proof in case \eqref{eq1.13}. The proof in case
\eqref{eq1.12} is a hybrid of the proofs in the other three cases.
We have tried to avoid repetition of arguments which occur in two
or more cases by giving them, without repetition, in the proofs of some
lemmas in Section \ref{sec2}.  Also, since the first few paragraphs of
the proofs in cases \eqref{eq1.10}--\eqref{eq1.12} are the same, we
have in Section \ref{sec3} presented them only once.

For simplicity and to more easily compare our results to those in
\cite{AS}, we stated in Theorems \ref{thm1.1.3} and \ref{thm1.1.4}
special cases of our more general results which are the following two
theorems and which address the nonexistence of positive solutions of
the inequality
\begin{equation}\label{eq1.3}
-\Delta^2 v\ge |y|^{-\sigma}f(v)
\end{equation}
in exterior domains in $\mathbb{R}^n$, $n\ge 2$.

\begin{thm}\label{thm1.2}
 Suppose $\sigma<2$ is a constant,
$\Omega$ is a compact subset of $\mathbb{R}^n$, $n\ge 3$, and
 $f:(0,\infty)\to(0,\infty)$ is a continuous function satisfying
 \begin{equation}\label{eq1.4}
\liminf_{s\to 0^+}\frac{f(s)}{s^{1+\frac{4-\sigma}{n-2}}}>0\quad\mbox{ and }\quad 
\lim_{s\to \infty}\frac{f(s)}{s^{-1+\frac{\sigma}{2}}}=\infty.
 \end{equation}
Then there does not exist a $C^4$ positive solution $v(y)$ of \eqref{eq1.3}
in $\mathbb{R}^n \setminus \Omega$.
\end{thm}

\begin{thm}\label{thm1.3}
 Suppose $\sigma\in[0,2)$ is a constant,
$\Omega$ is a compact subset of $\mathbb{R}^2$, and
 $f:(0,\infty)\to(0,\infty)$ is a continuous function satisfying
 \begin{equation}\label{eq1.6}
  \liminf_{s\to \infty} \frac{f(s)}{s^{-1+\frac{\sigma}{2}}}
\frac{\prod_{i=2}^k\log^is}{(\log s)^{1-\frac{\sigma}{2}}}>0
 \end{equation}
 for some integer $k\ge 2$ where $\log^2=\log\circ\log$,
 $\log^3=\log\circ\log\circ\log$, etc.  Then there does not exist a
 $C^4$ positive solution $v(y)$ of \eqref{eq1.3} in $\mathbb{R}^2
 \setminus \Omega$.
\end{thm}

\begin{rem}
  Theorems \ref{thm1.2} and \ref{thm1.3} with $\sigma=0$ immediately
  imply Theorems \ref{thm1.1.3} and \ref{thm1.1.4}, respectively.
\end{rem}

\begin{rem}
Similar to Remarks \ref{rem1}, \ref{rem2}, and \ref{rem4}, the
exponents $1+\frac{4-\sigma}{n-2}$ and $-1+\frac{\sigma}{2}$ in
\eqref{eq1.4} are optimal and so is the exponent $-1+\frac{\sigma}{2}$
in \eqref{eq1.6}.
\end{rem}

Mitidieri and Pohozaev \cite[Remark 9.1]{MP} have shown that the
problem
\[
\pm \Delta^mu\ge |x|^{-2m}|u|^q, \quad x\in\mathbb R^n\setminus \{0\},
\]
has no nontrivial weak solution if $m,n\ge 1$ and $q>1$.
Also, nonnegative solutions of
problems of the form
\begin{equation}\label{eq1.33}
 -\Delta^m u = f(x,u)\quad\text{or}\quad  -\Delta^mu \ge f(x,u)
\end{equation}
when $f$ is a nonnegative function have been studied in
\cite{CMS,CX,G,GL,H,MR,T,WX,X} and elsewhere. These problems arise
naturally in conformal geometry and in the study of the Sobolev
embedding of $H^{2m}$ into $L^{\frac{2n}{n-2m}}$.

Nonexistence results for entire solutions $u$ of problems
\eqref{eq1.33} can be used to obtain, via
scaling methods, estimates of solutions of boundary
value problems associated with \eqref{eq1.33}.
An excellent reference for polyharmonic boundary
value problems is \cite{GGS}.

Also, weak solutions of $\Delta^mu = \mu$, where
$\mu$ is a measure on a subset of ${\bb R}^n$, have been studied in
\cite{CDM,FKM,FM}, and removable isolated singularities of 
$\Delta^mu=0$ have been studied in \cite{H}.

\section{Preliminary results}\label{sec2}
In this section we provide some results needed for the proofs of Theorems
\ref{thm1.2} and \ref{thm1.3}.

\begin{lem}\label{lem2.1}
Suppose $m\ge 1$ and $n\ge 2$ are integers and $\Gamma(z)=\Gamma(|z|)$
is a radial solution of $\Delta^m \Gamma=0$ in
$\mathbb{R}^n\setminus\{0\}$. For each $r>0$, let
\begin{equation}\label{A1}
u(x;r)=\frac{1}{|\partial B_r |} \int_{|y|=r}
\Gamma(|x-y|)\,dS_y  
\quad\text{for } x\in \mathbb{R}^n.
\end{equation}
Then 
\begin{equation}\label{A2} 
u(x;r)=\begin{cases}
   \sum_{i=0}^{m-1} \frac{\Delta^i\Gamma(r)}{\alpha_i}|x|^{2i}, &\text{if } |x|<r \\
   \sum_{i=0}^{m-1} \frac{\Delta^i\Gamma(|x|)}{\alpha_i}r^{2i}, &\text{if } |x|>r
  \end{cases}
 \end{equation}
where $\alpha_0=1$ and
\[
\alpha_i=\Delta^i|x|^{2i}=i!2^i[n(n+2)(n+4)\cdots (n+2i-2)] 
\quad\text{for } i=1,2,\dots, m-1.  
\]
\end{lem}

\begin{proof}
Since $u(x;r)$ is radial in $x$, we can define
$v:[0,\infty)\times(0,\infty)\to\mathbb{R}$ by $v(|x|,r)=u(x;r)$ and
to prove Lemma \ref{lem2.1} it suffices to prove 
\begin{equation}\label{A3} 
v(\rho,r)=\begin{cases}
   \sum_{i=0}^{m-1} \frac{\Delta^i\Gamma(r)}{\alpha_i}\rho^{2i}, &\text{if } \rho<r \\
   \sum_{i=0}^{m-1} \frac{\Delta^i\Gamma(\rho)}{\alpha_i}r^{2i}, &\text{if } \rho>r.
  \end{cases}
\end{equation}
Since 
\begin{align*}
v(\rho,r)&=\frac{1}{|\partial B_\rho|}\int_{|x|=\rho} v(|x|,r)\,dS_x\\
&=\frac{1}{|\partial B_\rho|}\frac{1}{|\partial B_r|}
\int_{|x|=\rho} \int_{|y|=r}\Gamma(|x-y|)\,dS_y\,dS_x
\end{align*}
we see that
\begin{equation}\label{A4}
v(\rho,r)=v(r,\rho) \quad\text{for } (\rho,r)\in
(0,\infty)\times(0,\infty).
\end{equation}
Since $u(x;r)$ is a $C^\infty$ radial solution of $\Delta^mu=0$ in
$B_r(x)$ there are constants $c_i$ such that 
\[
u(x;r)=\sum_{i=0}^{m-1} c_i|x|^{2i} \quad\text{for } |x|<r. 
\]
Hence $(\Delta^ju)(0;r)=c_j\Delta^j|x|^{2j}=c_j\alpha_j$ for
$j=0,1,\dots,m-1$. On the other hand, it follows from \eqref{A1} that
$(\Delta^j u)(0;r)=\Delta^j\Gamma(r)$ and hence
\[
c_j=\frac{\Delta^j\Gamma(r)}{\alpha_j}\quad\text{for } j=0,1,2,\dots,m-1.
\]
Thus \eqref{A2}, and hence \eqref{A3}, holds for $|x|=\rho<r$ and by
\eqref{A4} we have \eqref{A3} also holds for $\rho>r$.
\end{proof}

\begin{lem}\label{lem2.2}
 Suppose $r\in (0,\frac{1}{4}]$ and $\alpha\geq 1$.  Then 
$$\left(\log\frac{e|y|}{r}\right)^\alpha -|y|^2  \geq \frac{1}{4} \left(\log
     \frac{e|y|}{r}\right)^\alpha \quad\text{for} \quad r\leq |y|\leq 1.$$
\end{lem}

\begin{proof}
\underline{Case I.}  Suppose $r\le |y|\le 1$ and $|y|\le 3/4$.  Then
\[
\frac{3}{4}\left(\log\frac{e|y|}{r}\right)^\alpha
\ge\frac{3}{4}(\log e)^\alpha
\ge |y|^2.
\]
\underline{Case II.}  Suppose $r\leq |y|\leq 1$ and $|y|\ge 3/4$.  Then
\[
\frac{3}{4}\left(\log\frac{e|y|}{r}\right)^\alpha
\ge\frac{3}{4}(\log 3e)^\alpha \ge \frac{3}{4}2^\alpha
\ge |y|^2.
\]
\end{proof}

\begin{lem}\label{lem2.3}
 Suppose $f:(0,\infty)\to(0,\infty)$ is a continuous function
 satisfying 
$$\lim_{s\to \infty}s^\alpha f(s)=\infty \quad\text{for some 
constant $\alpha>0$}.$$ 
Then there exists a
 continuous function $\hat{f}:(0,\infty) \to(0,\infty)$ such that
 $\hat{f}\leq f$ on $(0,\infty)$, $\hat{f}=f$ on $(0,1]$, $\hat{f}$ is
   decreasing on $[1,\infty)$, and
\begin{equation}\label{eq2.1}
\lim_{s\to \infty}s^\alpha\hat{f}(s)=\infty.
\end{equation}
\end{lem}

\begin{proof}
 Define $\hat{f}:(0,\infty)\to(0,\infty)$ by 
 \[
  \hat{f}(s)=
 \begin{cases}
  f(s), &\text{if } 0<s\leq 1\\
  \min_{1\leq \zeta \leq s} f(\zeta) , &\text{if }1\leq s<\infty.
 \end{cases}
 \]
Clearly $\hat{f}$ is continuous, $\hat{f}\leq f$, and $\hat{f}$ is
decreasing on $[1,\infty)$.  Let $M>1$.  Choose $s_0 >1$ such that
  $\zeta^\alpha f(\zeta)\geq M$ for $\zeta\ge s_0$. Choose $s_1>s_0$ such
  that $s_1^\alpha\hat{f}(s_0)\ge M$. Then for $s\geq s_1$ we have
\begin{align*}
 s^\alpha\hat{f}(s) &=s^\alpha \min \{ \hat{f}(s_0), \ \min_{s_0 \leq \zeta \leq s} f(\zeta) \} \\
 &\geq \min \{ s_1^\alpha \hat{f}(s_0), \ \min_{s_0 \leq \zeta \leq s} \zeta^\alpha f(\zeta) \} \geq M
\end{align*}
which proves \eqref{eq2.1}.
\end{proof}

\begin{lem}\label{lem2.4}
 Suppose $h$ is a solution of
 \begin{equation}\label{eq2.2}
  \Delta^2 h=0 \quad \text{in} \quad \overline{B_1(0)} \setminus \{0\} 
\subset \mathbb{R}^n ,\ n\geq 3.
 \end{equation}
 Then there exist constants $c_i$, $i=1,\ldots,5$, such that for
 $0<r<1$ we have 
 \begin{equation*}
  \int_{r<|x|<1} |x|^{-4} h(x)\,dx=
\begin{cases}c_1 r^{n-2} +c_2 r^{n-4} +c_3 \log r+c_4 r^{-2}+c_5 
&\text{if $n=3$ or $n\ge 5$}\\
c_1 r^{2} +c_2 \log r +c_3 (\log r)^2+c_4 r^{-2}+c_5 
&\text{if $n=4$}.
 \end{cases}
\end{equation*}
\end{lem}

\begin{proof}
 It follows from \eqref{eq2.2} that there exist constants $\hat{c}_i ,
 i=1,2,3,4$, such that for $0<\rho <1$ we have
 \[
\bar{h}(\rho):=\frac{1}{|\partial B_1|\rho^{n-1}} \int_{|x|=\rho}
 h(x)\,dS_x 
= 
\begin{cases}\hat{c}_1 \rho^2 + \hat{c}_2 + \hat{c}_3 \rho^{4-n} +
  \hat{c}_4 \rho^{2-n} &\text{if $n=3$ or $n\ge 5$}\\
\hat{c}_1 \rho^2 + \hat{c}_2 + \hat{c}_3 \log \rho + \hat{c}_4 \rho^{-2}
&\text{if $n=4$}.
\end{cases}
\]
 Thus
 \begin{align*}
 \int_{r<|x|<1} |x|^{-4} h(x)\,dx&= \int^{1}_{r} \rho^{-4} 
\left( \int_{|x|=\rho} h(x)\,dS_x \right) \,d\rho
=|\partial B_1|\int^{1}_{r} \rho^{n-5} \bar{h} (\rho) \,d\rho \\
  &=
\begin{cases}|\partial B_1|\int^{1}_{r} 
(\hat{c}_1 \rho^{n-3} + \hat{c}_2 \rho^{n-5} + 
\hat{c}_3 \rho^{-1} + \hat{c}_4 \rho^{-3}) \,d\rho&\text{if $n=3$ or $n\ge 5$}\\
|\partial B_1|\int^{1}_{r} 
(\hat{c}_1 \rho + \hat{c}_2 \rho^{-1} + 
\hat{c}_3 \rho^{-1}\log \rho + \hat{c}_4 \rho^{-3}) \,d\rho&\text{if $n=4$}
\end{cases}
\end{align*}
 from which we obtain Lemma \ref{lem2.4}.
\end{proof}

\begin{lem}\label{lem2.5}
Suppose $v\in L^1_{loc}(B)$, where $B=B_1(0)\subset \mathbb R^n$, $n\ge 2$.
If
\[
\Delta^2 v=0 \quad\text{in } {\cal D\,}' (B \setminus
\{0\})
\]
and 
\begin{equation}\label{eq2.3}
 \int_{|x|<r} |v(x)|\,dx=o(r^3)\quad\text{as}\quad r\to 0^{+} 
\end{equation}
then for some constant $a$ and some $C^\infty$ solution H of $\Delta^2
H=0$ in $B$ we have
\[
 v=a\Phi+H \quad \text{in}\quad B\setminus \{0\}
\]
where $\Phi$ is given by \eqref{eq1.10}--\eqref{eq1.13}.
\end{lem}

\begin{proof}
Since the support of $\Delta^2 v$ is a single point we have
$\Delta^2 v$ is a finite linear combination of the
delta function and its derivatives:
$$\Delta^2 v=\sum_{|\beta|\leq k} a_\beta D^\beta \delta \quad
\text{in}\quad {\cal D\,}' (B).$$
We now use a method of Brezis and Lions \cite{BL} to show $a_\beta
 =0$ for $|\beta|\geq 1$.  Choose $\varphi \in C^{\infty}_{0}
 (B)$ such that $(-1)^{|\beta|} (D^\beta \varphi)(0)=a_\beta$ for
 $|\beta|\leq k$.  Let $\varphi_\vp
 (x)=\varphi(\frac{x}{\vp})$.  Then, for $0<\vp <1$,
 $\varphi_\vp \in C^{\infty}_{0} (B)$, and
\begin{align*}
 \int v \Delta^2 \varphi_\vp &=(\Delta^2 v)(\varphi_\vp)
=\sum_{|\beta|\leq k} a_\beta (D^\beta \delta)\varphi_\vp \\
 &=\sum_{|\beta|\leq k} a_\beta (-1)^{|\beta|} \delta(D^\beta
 \varphi_\vp)
=\sum_{|\beta|\leq k} a_\beta (-1)^{|\beta|} (D^\beta \varphi_\vp)(0)\\
 &=\sum_{|\beta|\leq k} a_\beta (-1)^{|\beta|}\frac{1}{\vp^{|\beta|}} 
(D^\beta \varphi)(0)=\sum_{|\beta|\leq k} a^{2}_{\beta}\frac{1}{\vp^{|\beta|}} .
\end{align*}
On the other hand,
\begin{align*}
 \int v \Delta^2 \varphi_\vp 
 &=\int v(x)\frac{1}{\vp^4}\Delta^2 \varphi\left(\frac{x}{\vp}\right)\,dx\\
 &\leq\frac{C}{\vp^4} \int_{|x|<\vp}
 |v(x)|\,dx=o(\vp^{-1}) 
\quad \text{as}\quad \vp \to 0^{+} 
\end{align*}
by \eqref{eq2.3}.  Hence $a_\beta =0$ for $|\beta|\geq 1$ and
consequently, letting $a=a_0$, we have $\Delta^2 (v-a\Phi)=0$ in
${\cal D\,}'(B)$.  Thus the lemma follows from the fact that weakly
biharmonic functions are $C^\infty$. 
\end{proof}

\begin{lem}\label{lem2.6}
Suppose $u(x)$ is a $C^4$ positive solution of
\begin{equation}\label{2.4}
-\Delta^2 u\ge \alpha |x|^\frac{2\sigma-4n}{n-2} u(x)^\lambda \quad
\text{in}\quad \overline{B_1(0)}\setminus\{0\}
\subset\mathbb R^n,\ n\ge 3,
\end{equation}
where $\alpha>0$ and $\sigma<4$ are constants and
$\lambda=1+\frac{4-\sigma}{n-2}$. Then 
\begin{equation}\label{eq2.5}
\liminf_{r\to 0^+}\frac{\bar u(r)}{r^2J(r)}=0
\end{equation}
where $\bar u(r)$ is the average of $u$ on the sphere $|x|=r$ and 
\begin{equation*}
 J(r):= \int_{r<|y|<1} |y|^{2-n} (-\Delta^2 u(y))\,dy 
\quad \text{for}\quad 0<r<1.
\end{equation*}
\end{lem}

\begin{proof}
Suppose for contradiction that there exists $\vp,r_0\in(0,1)$ such
that $\bar u(r)\ge \vp r^2 J(r)$ for $0<r\le r_0$. Then, letting $C$
denote a positive constant whose value may change from line to line,
we have for $0<r\le r_0$ that
\begin{align*}
 -J'(r)&= \int_{|y|=r} |y|^{2-n} (-\Delta^2 u(y))\,dS_y\\
 &=r^{2-n} |\partial B_1| r^{n-1} (-\overline{\Delta^2 u}(r))\\
 &=C r(-\overline{\Delta^2 u}(r))\\
 &\geq Cr\left(r^{\frac{2\sigma-4n}{n-2}} (\bar{u}(r))^{\lambda}\right)\\
 &\geq C r^{1+\frac{2\sigma-4n}{n-2}} (\vp r^2 J(r))^\lambda \\
 &=Cr^{1+\frac{2\sigma-4n}{n-2}+2\lambda} J(r)^\lambda \\
 &=Cr^{-1} J(r)^\lambda .
\end{align*}
Consequently  $-J'(r)J(r)^{-\lambda} \geq
Cr^{-1}$ for $0<r\leq r_0$ which implies
\begin{equation*}
 \frac{1}{(\lambda-1)J(r_0)^{\lambda-1}} \geq \frac{1}{\lambda-1} \left[ \frac{1}{J(r_0)^{\lambda-1}}-\frac{1}{J(r)^{\lambda-1}} \right] \geq C\log\frac{r_0}{r} \to \infty
\end{equation*}
as $r\to 0^{+}$, a contradiction, which proves the lemma.
\end{proof}

\begin{lem}\label{lem2.7}
Suppose $u$ is a $C^4$ positive solution of 
\[
-\Delta^2 u>0 \quad \text{in}\quad
B_2(0)\setminus\{0\}\subset\mathbb{R}^n,\quad n\ge 2,
\]
such that
\begin{equation}\label{eq2.110}
\int_{|x|<1}\varphi(x)u(x)\,dx<\infty
\end{equation}
where $\varphi:\overline{B_1(0)}\setminus\{0\}\to (0,\infty)$ is a
continuous radial function satisfying 
\begin{equation}\label{eq2.115}
\int_{|x|<1}\varphi(x)\,dx=\infty.
\end{equation}
Then
\begin{equation}\label{eq2.120}
\int_{|x|<1}-\Delta^2 u(x)\,dx<\infty.
\end{equation}
\end{lem}

\begin{proof}
Let $F(\rho)=-\overline{\Delta^2u}(\rho)=-\Delta^2\bar{u}(\rho)$.
Then for some constants $c_1$and $c_2$ we have for $0<r\le 1$ that
\[
\Delta \bar u(r)=
\begin{cases} c_1+c_2\log\frac1r-(NF)(r)&\text{if $n=2$}\\
c_1+c_2r^{2-n}-(NF)(r)<0&\text{if $n\ge 3$}
\end{cases}
\]
where $(NF)(r)=\int_r^1s^{-(n-1)}\int_s^1\rho^{n-1} F(\rho)\,d\rho\,ds$.

Suppose for contradiction that \eqref{eq2.120} is false. Then
$\int_0^1\rho^{n-1} F(\rho)\,d\rho=\infty$ and hence as $r\to 0^+$ we have
\[
(NF)(r)>>\begin{cases}\log\frac1r& \text{if $n=2$}\\
r^{2-n} & \text{if $n\ge 3$.}
\end{cases}
\]
Thus for small $r>0$ we have $\Delta \bar u(r)<0$.
Hence the positivity of $\bar u$ implies $\bar u>\vp>0$ for small
$r>0$, which together with \eqref{eq2.110} and \eqref{eq2.115} 
gives a contradiction and
completes the proof of Lemma \ref{lem2.7}.
\end{proof}

\begin{lem}\label{lem2.8}
 There does not exist a positive \emph{radial} solution of
 \begin{equation}\label{eq2.13}
  -\Delta^2 v\ge |y|^{-\sigma}f(v) \quad\text{in}\quad \mathbb{R}^2\setminus B_{r_{0/2}} (0)
 \end{equation}
 where $r_0\geq 2$ and $\sigma\in[0,2)$ are constants
 and $f:(0,\infty) \to (0,\infty)$ is a
 continuous function such that
 \begin{equation}\label{eq2.14}
  \liminf_{s \to \infty} \frac{f(s)}{s^{-1+\frac{\sigma}{2}}} > 0.
 \end{equation}
\end{lem}

\begin{proof}
  Suppose for contradiction that $v(r)$ is a positive radial solution
  of \eqref{eq2.13}.  Let $F(r)=-(\Delta^2 v)(r)$ and $(NF)(r)=
  \int^{r}_{r_0} \frac{1}{s} \int^{s}_{r_0} \rho F(\rho)\, d\rho\,ds.$ Then
  for some constants $c_1,...,c_4$ we have
 \begin{equation}\label{eq2.15}
  v(r)=c_1 + c_2 \log r + c_3 r^2 + c_4 r^2 \log r-(N^2 F)(r).
 \end{equation}
 If $\int^{s}_{r_0} \rho F(\rho)\,d\rho \to \infty$ as $s\to \infty$
 then $(NF)(r) >> \log r$ as $r\to \infty$ and hence $(N^2 F)(r)>>r^2
 \log r$ as $r\to \infty$ which together with \eqref{eq2.15}
 contradicts the positivity of $v$.  Thus
 \begin{equation}\label{eq2.6}
  \int^{\infty}_{r_0} \rho F(\rho)\,d\rho < \infty.
 \end{equation}
 We claim that
 \begin{equation}\label{eq2.7}
  \liminf_{r\to \infty} v(r)=0.
 \end{equation}
 To see this, suppose for contradiction that \eqref{eq2.7} is false.
 Then for some $\vp >0$ we have $v(r)>\vp$ for $r_0 \leq
 r<\infty$.  Thus by \eqref{eq2.13}, \eqref{eq2.14} and \eqref{eq2.15} we
 have 
\begin{align*}
F(r)&=-\Delta^2 v(r)\geq r^{-\sigma}f(v(r)) 
\geq \frac{1}{Cr^\sigma v(r)^{1-\frac{\sigma}{2}}}\\
 &\geq \frac{1}{Cr^\sigma (r^2\log r)^{1-\frac{\sigma}{2}}}
\ge \frac{1}{Cr^2\log r} \quad\text{for $r$ large}
\end{align*}
which contradicts \eqref{eq2.6} and proves \eqref{eq2.7}.

 By \eqref{eq2.6}, 
\[
(\hat{N}F)(r):=\int^{r}_{r_0} \frac{1}{s}
 \int^{\infty}_{s} \rho F(\rho)\, d\rho\,ds =o(\log r) 
\quad\text{as}\quad r\to \infty
\]
and thus
 \begin{equation}\label{eq2.8}
  (N\hat{N}F)(r)=o(r^2 \log r) \quad\text{as}\quad r\to \infty.
 \end{equation}
 Since $v(r)$ solves \eqref{eq2.13}, there exist constants
 $\hat{c_1},...,\hat{c_4}$ such that
 \begin{equation}\label{eq2.9}
   v(r)=\hat{c}_1+ \hat{c}_2 \log r+\hat{c}_3 r^2 
+\hat{c}_4 r^2 \log r+(N\hat{N}F)(r).
 \end{equation}
 Since $v>0$, it follows from \eqref{eq2.8} and \eqref{eq2.9} that
 \begin{equation}\label{eq2.10}
  \hat{c}_4 \geq 0.
 \end{equation}
 If $(\hat{N}F)(r) \to \infty$ then $(N\hat{N}F)(r)>>r^2$ as $r\to
 \infty$ which together with \eqref{eq2.9} and \eqref{eq2.10} implies
 $v(r)\to \infty$ as $r\to\infty$ which contradicts \eqref{eq2.7}.  Hence
 $(\hat{N}F)(r)$ is bounded.  Thus
 \[
(\hat{\hat{N}}F)(r):=\int^{\infty}_{r} \frac{1}{s} \int^{\infty}_{s}
 \rho F(\rho)\,d\rho\,ds=o(1) \quad\text{as } r\to \infty
\] 
which implies
 \begin{equation}\label{eq2.11}
  (N\hat{\hat{N}}F)(r)=o(r^2) \quad\text{as}\quad r\to \infty.
 \end{equation}

 Since $v$ solves \eqref{eq2.13} there exist constants $\hat{\hat{c}}_1,
 ..., \hat{\hat{c}}_4$ such that
 \begin{equation}\label{eq2.12}
  v(r)=\hat{\hat{c}}_1 +\hat{\hat{c}}_2 \log r+\hat{\hat{c}}_3 r^2 
+\hat{\hat{c}}_4 r^2 \log r-(N\hat{\hat{N}}F)(r).
 \end{equation}

 By \eqref{eq2.11} and \eqref{eq2.12} and the positivity of $v$ we have
 $\hat{\hat{c}}_4 \geq 0$ and then by \eqref{eq2.7}, $\hat{\hat{c}}_4
 =0$.  Hence by \eqref{eq2.11} and \eqref{eq2.12} and the positivity of
 $v$ we have $\hat{\hat{c}}_3 \geq 0$ and then by \eqref{eq2.7},
 $\hat{\hat{c}}_3 =0$.  Thus by \eqref{eq2.12} 
\[
v(r)=\hat{\hat{c}}_1
 +\hat{\hat{c}}_2 \log r-(N\hat{\hat{N}}F)(r) 
\]
and so $-\Delta
 v=\hat{\hat{N}}F>0$ which together with the positivity of $v$
 contradicts \eqref{eq2.7} and completes the proof of Lemma \ref{lem2.8}.
\end{proof}

\begin{lem}\label{lem2.9}
  Suppose $x,y \in \mathbb{R}^2$ and $y\neq 0$.  Then
  \[
I(x,y):=\int^{1}_{0} (1-t) \log
  \frac{|y|}{|y-tx|}\, dt\leq 2 \int^{1}_{0} \log
  \frac{1}{s}\,ds< \infty.
\]
\end{lem}

\begin{proof}
 Since $I(0,y)=0$ we can assume $x\neq 0$.  Under the change of 
variables $\tau=\frac{|x|}{|y|} t$ we have 
 \begin{align*}
   I(x,y)&= \frac{|y|}{|x|} \int^{\frac{|x|}{|y|}}_{0} 
\left(1-\frac{|y|}{|x|}\,\tau \right)
\log \frac{1}{|\frac{y}{|y|}-\tau \frac{x}{|x|}|}\,d\tau \\
   &\leq \frac{|y|}{|x|} \int^{\frac{|x|}{|y|}}_{0} 
\left(1-\frac{|y|}{|x|}\tau \right)\log \frac{1}{|1-\tau|} \,d\tau \\
   &=\varphi(\frac{|x|}{|y|})
 \end{align*}
where $\varphi :(0,\infty) \to \mathbb{R}$ is given by
 \begin{align*}
 \varphi(\rho): &= \frac{1}{\rho} \int^{\rho}_{0} 
\left(1-\frac{\tau}{\rho} \right) \log \frac{1}{|1-\tau|} \,d\tau \\
 &\leq \frac{1}{\rho} \int^{\min\{\rho,2\}}_{0} \log \frac{1}{|1-\tau|} \,d\tau \\ 
 &\leq 
 \begin{cases}
  \frac{1}{2} \int^{2}_{0} \log \frac{1}{|1-\tau|} \,d\tau, &\text{if }\rho \geq2\\
  \frac{1}{\rho} \int^{\rho}_{0} \log \frac{1}{|1-\tau|} \,d\tau, 
&\text{if }0<\rho \le 2
 \end{cases}\\
 &\leq \int^{2}_{0} \log \frac{1}{|1-\tau|}\,d\tau
=2 \int^{1}_{0} \log \frac{1}{s} \,ds.
 \end{align*}
\end{proof}

\begin{lem}\label{lem2.10}
There does not exist a $C^4$ positive solution of 
\begin{equation}\label{eq2.20}
-\Delta^2v\ge v^{-1} \quad 
\text{in $\mathbb R^n\setminus \overline{B_{R/2}(0)}$, $n\ge 3$,}
\end{equation}
where $R$ is a positive constant.
\end{lem}

\begin{proof} 
By averaging \eqref{eq2.20} we can assume $v$ is radial.  Let
$F(r)=-\Delta^2v(r)$. Then 
\begin{equation}\label{eq2.21}
v(r)=c_1+c_2r^2+c_3r^{2-n}+c_4\Phi(r)-(N^2F)(r)\quad\text{for $r\ge R$}
\end{equation}
where $\Phi(r)$ is given by \eqref{eq1.10}--\eqref{eq1.12} and
\[
(NF)(r):=\int_R^rs^{1-n}\int_R^s\rho^{n-1}F(\rho)\,d\rho\,ds\ge 0.
\]
Thus for some positive constant $C$ we have $v(r)<Cr^2$ for $r\ge
R$, which implies 
\[
F(r)=-\Delta^2v(r)\ge v(r)^{-1} \ge \frac{1}{Cr^2} \quad \text{for
  $r\ge R$.}
\]
Hence $(NF)(r)\to\infty$ as $r\to\infty$. Thus $(N^2F)(r)>>r^2$ as
$r\to\infty$ which together with \eqref{eq2.21} contradicts the
positivity of $v(r)$.
\end{proof}

\section{Beginning of the Proof of Theorem \ref{thm1.2}} \label{sec3}
In this section we begin the proof of Theorem \ref{thm1.2}.
In Sections \ref{sec4},  \ref{sec5}, and \ref{sec6}, we will complete
the proof of Theorem \ref{thm1.2} when $n\ge 5$, $n=4$, and $n=3$,
respectively. 
\medskip

\n {\it Beginning of the proof of Theorem \ref{thm1.2}}.  
Suppose for contradiction
that $v(y)$ is a $C^4$ positive solution of \eqref{eq1.3} in
$\mathbb{R}^n\setminus\Omega$.  By scaling $v$, we can assume
$\Omega=\overline{B_{1/2}(0)}$ and
\begin{equation}\label{eq3.1}
  f(s)\ge s^{1+\frac{4-\sigma}{n-2}} \quad \text{for} \quad 0<s\leq 1.
 \end{equation}
Moreover, by Lemma \ref{lem2.3}, we can assume
\begin{equation}\label{eq3.2}
f \text{ is decreasing on }[1,\infty).
\end{equation}
Let $u(x)=|y|^{n-4} v(y)$, $y=\frac{x}{|x|^2}$ be the $2$-Kelvin
transform of $v(y)$.  Then
\[
v(y)=|x|^{n-4} u(x)\quad\text{and} \quad \Delta^2 v(y)=|x|^{n+4} \Delta^2 u(x). 
\]
(See \cite{WX} and \cite{X}.)  It follows therefore from \eqref{eq1.3}
and \eqref{eq3.1} that $u(x)$ is a $C^4$ positive solution of
\begin{equation}\label{eq3.3}
 -\Delta^2 u(x)\geq 
\begin{cases}
  |x|^{\frac{2\sigma-4n}{n-2}} u(x)^{1+\frac{4-\sigma}{n-2}} , &\text{if }0<u(x)\leq |x|^{4-n}\\
  |x|^{-n-4+\sigma} f(|x|^{n-4}u(x)) , &\text{if }u(x)\geq |x|^{4-n}
 \end{cases}
\quad\text{in }B_2 (0) \setminus \{0\}.
\end{equation}

Let $\Psi$ and $N$ be as defined in Appendix \ref{secA}. Since $u$
is a $C^4$ positive solution of
\eqref{eqA.1}, it follows from Theorem \ref{thmA} that $u$
satisfies \eqref{eq3.6}, \eqref{eq3.7}, and \eqref{eq3.8}.

By \eqref{eq3.8} and Lemma \ref{lem2.4}, there exist
constants $c_i, i=1,\ldots,5$, such that for $0<r<1$ we have
\begin{align}\notag
 \int_{r<|x|<1} |x|^{-4} u(x)\,dx=&
\int_{r<|x|<1} |x|^{-4} N(x)\,dx\\
&+\begin{cases}c_1r^{n-2} +c_2 r^{n-4} +c_3 \log \frac{1}{r} +c_4 r^{-2}+c_5
&\text{if $n=3$ or $n\ge 5$}\\\label{eq3.10}
c_1r^2 +c_2\log\frac{e}{r} +c_3 (\log \frac{e}{r})^2 +c_4 r^{-2}+c_5
&\text{if $n=4$.}
\end{cases}
\end{align}

\section{Completion of the Proof of Theorem \ref{thm1.2} when $n\ge  5$} 
\label{sec4}

When $n\ge 5$, we complete in this section the proof of Theorem
\ref{thm1.2} which we began in Section \ref{sec3}.

\begin{proof}[Completion of the proof of Theorem \ref{thm1.2} when
  $n\ge 5$] 
For $x\in \mathbb R^n$, $n\ge 5$, we see by Lemma \ref{lem2.1} that
\begin{equation}\label{eq3.0}
\frac{1}{|\partial B_r |} \int_{|y|=r} \frac{1}{|x-y|^{n-4}}\,dS_y
=\begin{cases}
   r^{4-n}-\frac{n-4}{n} r^{2-n} |x|^2 , &\text{if } |x|<r \\
   |x|^{4-n}-\frac{n-4}{n}r^2 |x|^{2-n}, &\text{if } |x|>r.
  \end{cases}
\end{equation}
It therefore follows from equations \eqref{eq1.10} and \eqref{eq3.5} 
that for $r>0$ we have
\[
 \frac{1}{A|\partial B_r |} \int_{|x|=r} \Psi (x,y)\,dS_x=
 \begin{cases}
  -\frac{n-4}{n}r^{2-n} |y|^2 , &\text{if }|y|<r \\
  -r^{4-n} p\left(\frac{r}{|y|}\right), &\text{if }|y|\geq r
 \end{cases}
\]
where $p(t):=1-t^{n-4}+\frac{n-4}{n}t^{n-2}$ is bounded between
positive constants for $0\leq t \leq 1$.  Hence
\[
 -\int_{|x|=r} \Psi(x,y)\,dS_x \sim 
 \begin{cases}
  2r|y|^2, &\text{if }|y|<r\\
  r^3, &\text{if }|y|\geq r
 \end{cases}
\quad\text{ for }\quad (r,y) \in (0,\infty) \times \mathbb{R}^n .
\]
(If $f$ and $g$ are nonnegative functions defined on a set $S$ then
when we write ``$f(X) \sim g(X)$ for $X \in S$'' we mean there
exist positive constants $C_1$ and $C_2$ such that $C_1 g(X) \leq f(X)
\leq C_2 g(X)$ for all $X \in S$.)  Thus by \eqref{eq3.9}, for $0<r\leq
\frac{1}{4}$, we have
\begin{align}
 \notag \int_{r<|x|<1} |x|^{-4} N(x)\,dx =& \int_{|y|<1} \int^{1}_{r} 
\rho^{-4} \int_{|x|=\rho} -\Psi(x,y)\,dS_x \,d\rho (-\Delta^2 u(y))\,dy\\
\notag \sim& \int_{r<|y|<1} \left( \int^{1}_{|y|} 2\rho^{-3} |y|^2\,d\rho +
\int^{|y|}_{r} \rho^{-1} \,d\rho \right) (-\Delta^2 u(y))\,dy\\
\notag &+ \int_{|y|<r}\left(\int_r^12\rho^{-3}|y|^2\,d\rho\right)(-\Delta^2 u(y))\,dy\\
\sim& \int_{r<|y|<1} \left(\log \frac{e|y|}{r}\right)(-\Delta^2 u(y))\,dy+g(r)\label{eq3.11}
\end{align}
by Lemma \ref{lem2.2} with $\alpha=1$ where
\begin{equation}\label{eq3.12}
 0<g(r):=\left(\frac{1}{r^2}-1\right)\int_{|y|<r} |y|^2 (-\Delta^2
 u(y))\,dy=o\left(\frac{1}{r^2}\right) \quad
\text{as}\quad r\to0^{+}
\end{equation}
by \eqref{eq3.7}.

Let $\varphi(t,r)=t^{-2}\log \frac{et}{r}$.  Since $\varphi_t (t,r)<0$
for $t\geq r>0$, we see for $0<r<1$ that 
\begin{align*}
 \int_{r<|y|<1} \left(\log \frac{e|y|}{r}\right)&(-\Delta^2
  u(y))\,dy=\int_{r<|y|<1} \varphi(|y|,r)|y|^2 (-\Delta^2 u(y))\,dy\\
 &\leq \varphi(r,r) \int_{r<|y|<\sqrt{r}} |y|^2(-\Delta^2 u(y))\,dy+
\varphi(\sqrt{r},r) \int_{\sqrt{r}<|y|<1} |y|^2 (-\Delta^2 u(y))\,dy\\
 &\leq \frac{1}{r^2} \int_{r<|y|<\sqrt{r}} |y|^2(-\Delta^2 u(y))\,dy
+\frac{1}{r} \left(\log \frac{e}{\sqrt{r}}\right) \int_{|y|<1} |y|^2 (-\Delta^2 u(y))\,dy\\
 &=o(r^{-2})\quad \text{as}\quad r\to 0^{+}
\end{align*}
by \eqref{eq3.7}. It follows therefore from \eqref{eq3.11} and
\eqref{eq3.12} that
$$\int_{r<|x|<1} |x|^{-4}N(x)\,dx=o(r^{-2})\quad \text{as}\quad r\to0^{+}.$$ 
Hence, by \eqref{eq3.10} and the positivity of $u$ we see
that the constant $c_4$ in \eqref{eq3.10} is nonnegative and thus by
\eqref{eq3.10}, \eqref{eq3.11}, and the positivity of $g$ we have
\begin{equation}\label{eq3.13}
 \int_{r<|x|<1} |x|^{-4} u(x)\,dx \geq \frac{1}{C} \int_{r<|y|<1} \left(\log
 \frac{e|y|}{r}\right)(-\Delta^2 u(y))\,dy-C\log\frac{e}{r}\quad \text{for}
\quad 0<r<\frac{1}{4}
\end{equation}
where $C$ is a positive constant independent of $r$.  

By \eqref{eq3.6} there exists a constant $M>1$ such that $0<u(x)\leq
M|x|^{2-n}$ for $0<|x|\leq 1$. Define $I_1, I_2 :(0,1)\to[0,\infty)$
by
$$I_1(r):=M\int_{x\in S_1(r)} |x|^{-n-2}\,dx\quad \text{and}\quad I_2(r):=\int_{x\in S_2 (r)} |x|^{-4}u(x)\,dx$$ where 
\[S_1(r):=\{ x\in \mathbb{R}^n :r<|x|<1 \text{ and }|x|^{4-n}
<u(x)\leq M|x|^{2-n} \}\] 
and 
\[S_2(r):=\{ x\in \mathbb{R}^n :r<|x|<1\text{ and }0<u(x)\leq |x|^{4-n} \}.\]  Then
$S_1(r) \cup S_2(r)=B_1(0)-\overline{B_r(0)}$,
\begin{equation}\label{eq3.14}
 I_2(r)=O\left(\log\frac{1}{r}\right)\quad \text{as}\quad r\to0^{+} ,
\end{equation}
and for $0<r<\frac{1}{4}$ we have
\begin{align}\label{eq3.15}
 I_1(r)+I_2(r) &\geq \int_{r<|x|<1} |x|^{-4} u(x)\,dx\\
&\geq \frac{1}{C} \int_{r<|y|<1} \left(\log \frac{e|y|}{r}\right)(-\Delta^2 u(y))\,dy-C\log\frac{e}{r}
\label{eq3.16}\end{align}
by \eqref{eq3.13}.

By \eqref{eq3.3} we have
\begin{equation}\label{eq3.17}
 \int_{r<|y|<1} \left(\log\frac{e|y|}{r}\right)(-\Delta^2 u(y))\,dy\geq J_1(r)+J_2(r)\quad
\text{for}\quad 0<r<1
\end{equation}
where 
\[J_2(r):=\int_{S_2(r)}|y|^{\frac{2\sigma-4n}{n-2}}u(y)^{1+\frac{4-\sigma}{n-2}}\,dy\] 
and 
\[
J_1(r):= \int_{S_1(r)} \left(\log \frac{e|y|}{r}\right)|y|^{-n-4+\sigma} f(|y|^{n-4} u(y))\,dy.
\]

Before continuing with the proof of Theorem \ref{thm1.2}, we prove the
following lemma.

\begin{lem}\label{lem3.1}
As $r\to 0^+$ we have
\begin{equation}\label{eq3.17.5}
J_1(r)=O\left(\log\frac{1}{r}\right),
\end{equation}
\begin{equation}\label{eq3.18}
J_2(r)=O\left(\log\frac{1}{r}\right),
\end{equation}
and
\begin{equation}\label{eq3.19}
I_1(r)=o\left(\log\frac{1}{r}\right).
\end{equation}
\end{lem}
\begin{proof}
By \eqref{eq3.17}, \eqref{eq3.16}, and \eqref{eq3.14} we 
have  
\begin{align}\notag
 J_1(r) &\leq J_1(r)+J_2(r) \leq C \left[ I_1(r)+I_2(r)+\log\frac{e}{r} \right]\\
 \label{eq3.19.5}
&=CI_1(r)+O(\log\frac{1}{r})\quad\text{ as }\quad r\to 0^+.
\end{align}
If $S_1(0)=\emptyset$ then $I_1(r)\equiv J_1(r)\equiv 0$ for $0<r<1$
and thus \eqref{eq3.18} follows from \eqref{eq3.19.5}. Hence we can
assume $S_1(0)\not=\emptyset$. So for $r$ small and positive,
$S_1(r)\not=\emptyset$, $I_1(r)>0$, and 
\begin{align*}
 J_1(r)&\geq \left[ \inf_{y\in S_1(r)} \left(\log
   \frac{e|y|}{r}\right)|y|^{-2+\sigma}f(|y|^{n-4} u(y)) \right] 
\frac{I_1(r)}{M}\\
 &\geq \left[ \inf_{r<|y|<1} \left(\log\frac{e|y|}{r}\right)
(|y|^{-2})^{1-\sigma/2}f(M|y|^{-2})
  \right] 
\frac{I_1(r)}{M}
\end{align*}
by \eqref{eq3.2} and because $1< |y|^{n-4} u(y)\leq M|y|^{-2}$ for
$y\in S_1(r)$.  Thus 
\begin{equation*}
 \frac{M^{2-\sigma/2} J_1 (r)}{I_1(r)} \geq \min \left\{\ \inf_{r<|y|<\sqrt{r}}
 (M|y|^{-2})^{1-\sigma/2}f(M|y|^{-2}),\ \left(\log\frac{e}{\sqrt{r}}\right) 
\inf_{\sqrt{r}<|y|<1} (M|y|^{-2})^{1-\sigma/2}f(M|y|^{-2}) \right\}\
 \to \infty 
\end{equation*}
as $r\to 0^+$ by \eqref{eq1.4}.  
Hence \eqref{eq3.19.5} implies \eqref{eq3.17.5}; and \eqref{eq3.17.5}
implies \eqref{eq3.19}. Finally, \eqref{eq3.18} follows from
\eqref{eq3.19.5} and \eqref{eq3.19}.
\end{proof}

Continuing with the proof of Theorem \ref{thm1.2}, it follows from
\eqref{eq3.14}, \eqref{eq3.19}, and \eqref{eq3.16} that there is a
constant $C>0$ such that for $0<r<1$ we have
\begin{align*}
 C&\geq \frac{1}{\log\frac{e}{r}} \int_{r<|y|<1} \left(\log\frac{e|y|}{r}\right)(-\Delta^2 u(y))\,dy\\
 &\geq \frac{1}{\log\frac{e}{r}} \int_{\sqrt{r}<|y|<1} \left(\log\frac{e}{\sqrt{r}}\right)(-\Delta^2 u(y))\,dy\\
 &\geq \frac{1}{2} \int_{\sqrt{r}<|y|<1} -\Delta^2 u(y)\,dy.
\end{align*}
Thus
\begin{equation}\label{eq3.20}
 \int_{|y|<1} -\Delta^2 u(y)\,dy<\infty.
\end{equation}

By \eqref{eq3.19} there exists a constant $C>0$ such that
$I_1(2^{-(j+1)}) \leq C(j+1)$ for $j=0,1,2,\ldots$.  Thus for each
$\vp>0$ we have
\begin{align*}
 M \int_{S_1(0)}|x|^{-n-2+\vp}\,dx&\leq
 \sum^{\infty}_{j=0}2^{-\vp j}M 
\int_{\substack{2^{-(j+1)}<|x|<2^{-j}\\ x\in S_1(0)}} |x|^{-n-2} \,dx\\
 &\leq \sum^{\infty}_{j=0} 2^{-\vp j} I_1 (2^{-(j+1)})\\
 &\leq C\sum^{\infty}_{j=0}(j+1)(2^{-\vp})^j <\infty.
\end{align*}
Hence, for $0<\vp \leq 1$,
\[
 \int_{|x|<1} |x|^{-4+\vp} u(x)\,dx 
\leq M \int_{S_1(0)} |x|^{-n-2+\vp} \,dx+\int_{|x|<1} |x|^{-n+\vp} \,dx
 <\infty
\]
and so taking $\vp=1$ we have for $0<r<1$ that
\begin{equation}\label{eq3.21}
 \int_{|x|<r} u(x)\,dx \leq \int_{|x|<r} \frac{r^3}{|x|^3} u(x)\,dx=o(r^3)
 \quad\text{as}\quad r\to 0^+.
\end{equation}

Let $F=\Delta^2 u$ and 
$$\hat{N}(x)=\int_{B_1(0)} 
\Phi(x-y)F(y)\,dy\quad \text{for}\quad x\in \mathbb{R}^n .$$  
By \eqref{eq3.0} and \eqref{eq3.20} we have for $0<r\le 1$ that
\begin{align}
 \notag \int_{|x|<r} |\hat{N}(x)|\,dx &= \int_{|y|<1} \left( \int_{|x|<r} \frac{A}{|x-y|^{n-4}} \,dx \right) (-F(y))\,dy\\
 \notag &\leq \int_{|y|<1} \left( \int_{|x|<r} \frac{A}{|x|^{n-4}} \,dx \right) (-F(y))\,dy\\
 &=Cr^4 .\label{eq3.22}
\end{align}
In particular $\hat{N} \in L^{1}_{\text{loc}} (B_1(0))$.  
Also for $\varphi \in C^{\infty}_{0} (B_1(0))$ we have
\begin{align*}
 \int_{B_1(0)} \hat{N} \Delta^2 \varphi \,dx &= \int_{B_1(0)} \left( \int_{B_1(0)} \Phi(x-y)\Delta^2 \varphi (x)\,dx \right) F(y)\,dy\\
 &=\int_{B_1(0)} \varphi (y)F(y)\,dy.
\end{align*}
Thus $\Delta^2 \hat{N}=F$ in ${\cal D\,}'(B_1(0))$.

Let $v=u-\hat{N}$.  By \eqref{eq3.21} and \eqref{eq3.22} we see that
$v$ satisfies \eqref{eq2.3}.
Since $\Delta^2 u=F$ in ${\cal D\,}' (B_1(0) \setminus \{0\})$ we have
$\Delta^2 v=\Delta^2 u-\Delta^2 \hat{N}=0$ in ${\cal D\,}' (B_1(0) \setminus
\{0\})$.  
Thus Lemma \ref{lem2.5} implies for some constant $a$ and some
$C^\infty$ solution $H$ of
$\Delta^2 H=0$ in $B_1(0)$ we have
\begin{equation}\label{eq3.24}
 u=\hat{N}+a\Phi+H \quad \text{in}\quad B_1 (0)\setminus \{0\}.
\end{equation}
Hence, since $\hat{N}<0$ and $u>0$ we have $a\geq 0$.
\medskip

\noindent \underline{Case I.} Suppose the constant $a$ in
\eqref{eq3.24} is positive.  By \eqref{eq3.0} we have
\begin{equation*}
 \frac{\rho^{-4}}{A|\partial B_1 |} \int_{|x|=\rho} \Phi(x-y)\,dS_x \leq 
 \begin{cases}
  \rho^{-1}, &\text{if }|y|<\rho\\
  |y|^{4-n} \rho^{n-5} , &\text{if }|y|>\rho.
 \end{cases}
\end{equation*}
Thus for $0<r<1$ we have
\begin{align*}
 \frac{1}{A|\partial B_1 |} \int_{r<|x|<1} |x|^{-4} (-\hat{N}(x))\,dx &= \int_{|y|<1} \int^{1}_{r}\frac{\rho^{-4}}{A|\partial B_1 |} \int_{|x|=\rho} \Phi(x-y)\,dS_x \,d\rho(-\Delta^2 u(y))\,dy\\
 &\leq \int_{r<|y|<1} \left( \int^{|y|}_{r} |y|^{4-n} \rho^{n-5} \,d\rho+\int^{1}_{|y|} \rho^{-1}\,d\rho \right) (-\Delta^2 u(y))\,dy\\
 &+ \int_{|y|<r} \left( \int^{1}_{r} \rho^{-1} \,d\rho \right)(-\Delta^2 u(y))\,dy\\
 &\leq \int_{r<|y|<1} \left(\log\frac{e}{|y|}\right)(-\Delta^2 u(y))\,dy
+\left(\log\frac{1}{r}\right) \int_{|y|<r} -\Delta^2 u(y)\,dy\\
 &=o\left(\log\frac{e}{r}\right)\quad \text{as}\quad r\to 0^{+}
\end{align*}
by \eqref{eq3.20} and the fact that
\begin{align*}
 \int_{r<|y|<1} \left(\log\frac{e}{|y|}\right)(-\Delta^2 u(y))\,dy 
&\leq\left(\log\frac{e}{r}\right) \int_{r<|y|<(\log\frac{1}{r})^{-1}} -\Delta^2 u(y)\,dy\\
 &+\left(\log\left(e \log\frac{1}{r}\right)\right) \int_{|y|<1} -\Delta^2 u(y)\,dy.
\end{align*}
Hence by \eqref{eq3.24}
\begin{align*}
 \int_{r<|x|<1} |x|^{-4} u(x)\,dx &=a \int_{r<|x|<1}
 |x|^{-4}\frac{A}{|x|^{n-4}} \,dx
-o\left(\log\frac{1}{r}\right)\\
 &\geq C\log\frac{1}{r} \quad \text{for small }r>0
\end{align*}
where $C$ is a positive constant.  Thus by \eqref{eq3.15}
and \eqref{eq3.19} we have
\begin{equation}\label{eq3.25}
 I_2(r)\geq C\log\frac{1}{r}\quad \text{for small }r>0.
\end{equation}

On the other hand, by H\"older's inequality and \eqref{eq3.18},
\begin{align*}
 I_2(r)&=\int_{S_2(r)} |x|^{\frac{2\sigma-8}{n+2-\sigma}}
\left(|x|^{\frac{2\sigma-4n}{n+2-\sigma}} u(x)\right)\,dx\\
 &\leq \left( \int_{0<|x|<1} |x|^{-2} \,dx
 \right)^{\frac{4-\sigma}{n+2-\sigma}} 
J_2 (r)^{\frac{n-2}{n+2-\sigma}}\\
 &=O\left(\left(\log\frac{1}{r}\right)^{\frac{n-2}{n+2-\sigma}}\right)\quad\text{as}\quad r\to 0^{+}
\end{align*}
which contradicts \eqref{eq3.25} and completes the proof of Theorem
\ref{thm1.2} in Case I.
\medskip

\noindent \underline{Case II.} Suppose the constant $a$ in
\eqref{eq3.24} is zero.  Then
\begin{equation}\label{eq3.26}
 0<u=\hat{N}+H \quad \text{for}\quad 0<|x|<1.
\end{equation}
Thus $-\hat{N}$ and $u$ are positive and bounded for $0<|x|\le 1/2$
and so \eqref{eq3.3} implies
\begin{equation}\label{eq3.26.5}
-\Delta^2 u\ge C|x|^\frac{2\sigma-4n}{n-2} u^{1+\frac{4-\sigma}{n-2}} \quad\text{in}
\quad \overline{B_1(0)}\setminus\{0\}
\end{equation}
for some positive constant $C$. 
Also, since  \eqref{eq3.0} implies
\begin{align}\label{eq3.27}
 -\frac{1}{A}\bar{\hat{N}}(r): &=\frac{1}{A|\partial B_r|} \int_{|x|=r} -\hat{N} (x)\,dS_x =\int_{|y|<1} \frac{1}{|\partial B_r|} \int_{|x|=r} \frac{\,dS_x}{|x-y|^{n-4}} (-F(y))\,dy\\
 \notag &\geq \frac{4}{n} \int_{r<|y|<1} |y|^{4-n}
 (-F(y))\,dy\quad\text{for} \quad 0<r<1/2
\end{align}
we see that
\begin{equation}\label{eq3.28}
 -\hat{N}_0 :=\int_{|y|<1} \frac{A}{|y|^{n-4}} (-F(y))\,dy<\infty.
\end{equation}
Averaging \eqref{eq3.26} we obtain
\begin{equation}\label{eq3.29}
 0<\bar{u}(r)=\overline{(\hat{N}-\hat{N}_0 )}(r)+a_0 -a_1 Ar^2 
\quad \text{for}\quad 0<r<1
\end{equation}
for some constants $a_0$ and $a_1$.

By \eqref{eq3.27}, \eqref{eq3.28} and \eqref{eq3.0} we have
\begin{align}
 \notag &\frac{1}{A}\overline{(\hat{N}-\hat{N}_0 )}(r)\\
 \notag &= \int_{|y|<1} \left( \frac{1}{|\partial B_r |} \int_{|x|=r} \frac{\,dS_x}{|x-y|^{n-4}} -|y|^{4-n} \right) F(y)\,dy\\
 \label{eq3.30} &=\int_{|y|<r} \left(r^{4-n} -|y|^{4-n} -\frac{n-4}{n} r^{2-n}|y|^2\right)F(y)\,dy
-\int_{r<|y|<1} \frac{n-4}{n}r^2 |y|^{2-n} F(y)\,dy\\
 \notag &=-\int_{|y|<r} |y|^{4-n} \left[ 1-\left(\frac{|y|}{r}\right)^{n-4} 
+\frac{n-4}{n}\left(\frac{|y|}{r}\right)^{n-2} \right] F(y)\,dy\\
 \notag &-\int_{r<|y|<\sqrt{r}} \frac{n-4}{n}\left(\frac{r}{|y|}\right)^2 |y|^{4-n} F(y)\,dy-
\int_{\sqrt{r}<|y|<1} \frac{n-4}{n} \left(\frac{r}{|y|}\right)^2 |y|^{4-n} F(y)\,dy\\
 \notag &\to0 \quad \text{as}\quad r\to 0^{+}
\end{align}
by \eqref{eq3.28}.  Thus by \eqref{eq3.29} we have $a_0 \geq 0$.  
If $a_0>0$ then by \eqref{eq3.3} and the boundedness of $u$ in $B_1 (0)$ we have
\begin{align*}
 -\bar{F}(r)&=-\overline{\Delta^2 u}(r)\geq r^{\frac{2\sigma-4n}{n-2}} 
\bar{u}(r)^{1+\frac{4-\sigma}{n-2}}\\
 &\geq r^{-4-\frac{8-2\sigma}{n-2}} \left(\frac{a_0}{2}\right)^{1+\frac{4-\sigma}{n-2}}
\end{align*}
for $r$ small and positive and thus for small $r_0>0$ we have
\begin{align*}
 \int_{|y|<r_0} |y|^{4-n} (-F(y))\,dy &= \int^{r_0}_{0} r^{4-n} \int_{|y|=r} -F(y)\,dS_y \,dr\\
 &=|\partial B_1|\int^{r_0}_{0} r^3 (-\bar{F}(r))\,dr=\infty
\end{align*}
which contradicts \eqref{eq3.28}.  So $a_0=0,\bar{u}(r)\to 0$ as $r\to0$ 
and by \eqref{eq3.29} and \eqref{eq3.30} we have
\begin{align}
 \notag \frac{\bar{u}(r)}{A} &= \frac{\overline{(\hat{N}-\hat{N_0})}(r)}{A} -a_1 r^2\\
 \label{eq3.31}&= \int_{|y|<r} |y|^{4-n} \left[1-\left(\frac{|y|}{r}\right)^{n-4} 
+\frac{n-4}{n} \left(\frac{|y|}{r}\right)^{n-2} \right] (-\Delta^2 u(y))\,dy+(J(r)-a_1 )r^2
\end{align}
where
\begin{equation*}
 J(r):= \int_{r<|y|<1} \frac{n-4}{n} |y|^{2-n} (-\Delta^2 u(y))\,dy 
\quad \text{for}\quad 0\leq r<1.
\end{equation*}
($J(0)$ may be $\infty$.)
\medskip

\noindent \underline{Case II(a).} Suppose $a_1 <J(0)$.  Then there
exists $\vp>0$ and $r_0 \in (0,1)$ such that $a_1 \leq
(1-\vp)J(r)$ for $0<r\leq r_0$.  Thus by \eqref{eq3.31},
$\frac{\bar{u}(r)}{A} \geq \vp r^2 J(r)$ for $0<r<r_0$, which together
with \eqref{eq3.26.5} and Lemma \ref{lem2.6} gives a contradiction and
thereby proves Theorem \ref{thm1.2} when $n\ge 5$ in Case II(a).
\medskip

\noindent\underline{Case II(b).} Suppose $a_1 \geq J(0)$.  Then for
$0<r<1$ we have
\begin{equation*}
 (J(r)-a_1)r^2 \leq (J(r)-J(0))r^2 =-\int_{|y|<r} 
\frac{n-4}{n} r^2 |y|^{2-n} (-\Delta^2 u(y))\,dy
\end{equation*}
and hence by \eqref{eq3.31} we have
\begin{align}
 \notag 0<\frac{\bar{u}(r)}{A} &\leq \int_{|y|<r} |y|^{4-n} \left[ 1-\left(\frac{r}{|y|}\right)^{4-n} +\frac{n-4}{n} \left(\left(\frac{r}{|y|}\right)^{2-n} -\left(\frac{r}{|y|}\right)^2\right)\right] (-\Delta^2 u(y))\,dy\\
 \label{eq3.32}&=-\int_{|y|<r} |y|^2 r^{2-n} p\left(\frac{r}{|y|}\right)(-\Delta^2 u(y))\,dy
\end{align}
where $p(t):=\frac{n-4}{n} t^n -t^{n-2} +t^2-\frac{n-4}{n}$.  
Since $p(1)=p'(1)=p''(1)=0$ and
\[
 p'''(t) =(n-4)(n-2)(n-1)t^{n-5} \left(t^2 -\frac{n-3}{n-1}\right)>0
\quad\text{for}\quad t\geq 1
\]
we see that $p(t)>0$ for $t>1$.  This contradicts \eqref{eq3.32} 
and completes the proof of Theorem \ref{thm1.2} when $n\ge 5$ in
all cases.
\end{proof}

\section{Completion of the Proof of Theorem \ref{thm1.2} when $n=4$} 
\label{sec5}

When $n=4$, we complete in this section the proof of Theorem
\ref{thm1.2} which we began in Section \ref{sec3}.

\begin{proof}[Completion of the proof of Theorem \ref{thm1.2} when
  $n=4$] 
For $x\in \mathbb R^4$ we see by Lemma \ref{lem2.1} that
\begin{equation}\label{eq5.0}
\frac{1}{|\partial B_r |} \int_{|y|=r} \log\frac{e}{|x-y|}\,dS_y=
\begin{cases}
   \log\frac{e}{r} - \frac{1}{4}r^{-2}|x|^2, &\text{if } |x|<r \\
   \log\frac{e}{|x|} - \frac{1}{4}r^2|x|^{-2}, &\text{if } |x|>r.
  \end{cases}
 \end{equation}
It therefore follows from equations \eqref{eq1.11} and \eqref{eq3.5} 
that for $r>0$ we have
\[
 \frac{1}{A|\partial B_r |} \int_{|x|=r} \Psi (x,y)\,dS_x=
 \begin{cases}
  -\frac{1}{4}r^{-2} |y|^2 , &\text{if }|y|<r \\
  \left(\log\frac{r}{e|y|}\right) p\left(\frac{r}{|y|}\right), 
  &\text{if }|y|\geq r
 \end{cases}
\]
where $p(t):=((\log t)-t^2/4)/\log (t/e)$ is bounded between
positive constants for $0< t \leq 1$.  Hence
\[
 -\int_{|x|=r} \Psi(x,y)\,dS_x \sim 
 \begin{cases}
  2r|y|^2, &\text{if }|y|<r\\
  2r^3\log\frac{e|y|}{r}, &\text{if }|y|\geq r
 \end{cases}
\quad\text{ for }\quad (r,y) \in (0,\infty) \times \mathbb{R}^n .
\]
Thus by \eqref{eq3.9}, for $0<r\leq
\frac{1}{4}$, we have
\begin{align}
 \notag \int_{r<|x|<1} |x|^{-4} N(x)\,dx =& \int_{|y|<1} \int^{1}_{r} 
\rho^{-4} \int_{|x|=\rho} -\Psi(x,y)\,dS_x \,d\rho (-\Delta^2 u(y))\,dy\\
\notag \sim& \int_{r<|y|<1} \left( \int^{1}_{|y|} 2\rho^{-3} |y|^2\,d\rho +
\int^{|y|}_{r} 2\rho^{-1}\log\frac{e|y|}{\rho} \,d\rho \right) (-\Delta^2 u(y))\,dy\\
\notag &+ \int_{|y|<r}\left(\int_r^12\rho^{-3}|y|^2\,d\rho\right)(-\Delta^2 u(y))\,dy\\
\sim& \int_{r<|y|<1} \left(\log \frac{e|y|}{r}\right)^2(-\Delta^2 u(y))\,dy+g(r)\label{eq5.11}
\end{align}
by Lemma \ref{lem2.2} with $\alpha=2$ where
\begin{equation}\label{eq5.12}
 0<g(r):=\left(\frac{1}{r^2}-1\right)\int_{|y|<r} |y|^2 (-\Delta^2
 u(y))\,dy=o\left(\frac{1}{r^2}\right) \quad
\text{as}\quad r\to0^{+}
\end{equation}
by \eqref{eq3.7}.

Let $\varphi(t,r)=t^{-2}\left(\log \frac{et}{r}\right)^2$.  
Since $\varphi_t (t,r)\le 0$
for $t\geq r>0$, we see for $0<r<1$ that 
\begin{align*}
 \int_{r<|y|<1} \left(\log \frac{e|y|}{r}\right)^2&(-\Delta^2
  u(y))\,dy=\int_{r<|y|<1} \varphi(|y|,r)|y|^2 (-\Delta^2 u(y))\,dy\\
 &\leq \varphi(r,r) \int_{r<|y|<\sqrt{r}} |y|^2(-\Delta^2 u(y))\,dy+
\varphi(\sqrt{r},r) \int_{\sqrt{r}<|y|<1} |y|^2 (-\Delta^2 u(y))\,dy\\
 &\leq \frac{1}{r^2} \int_{r<|y|<\sqrt{r}} |y|^2(-\Delta^2 u(y))\,dy
+\frac{1}{r} \left(\log \frac{e}{\sqrt{r}}\right)^2 \int_{|y|<1} |y|^2 (-\Delta^2 u(y))\,dy\\
 &=o(r^{-2})\quad \text{as}\quad r\to 0^{+}
\end{align*}
by \eqref{eq3.7}. It follows therefore from \eqref{eq5.11} and
\eqref{eq5.12} that
$$\int_{r<|x|<1} |x|^{-4}N(x)\,dx=o(r^{-2})\quad \text{as}\quad r\to0^{+}.$$ 
Hence, by \eqref{eq3.10} and the positivity of $u$ we see
that the constant $c_4$ in \eqref{eq3.10} is nonnegative and thus by
\eqref{eq3.10}, \eqref{eq5.11}, and the positivity of $g$ we have
\begin{equation}\label{eq5.13}
 \int_{r<|x|<1} |x|^{-4} u(x)\,dx \geq \frac{1}{C} \int_{r<|y|<1}
 \left(\log \frac{e|y|}{r}\right)^2(-\Delta^2
 u(y))\,dy-C\left(\log\frac{e}{r}\right)^2\quad \text{for} \quad
 0<r<\frac{1}{4}
\end{equation}
where $C$ is a positive constant independent of $r$.  

By \eqref{eq3.6} there exists a constant $M>1$ such that $0<u(x)\leq
M|x|^{-2}$ for $0<|x|\leq 1$. Define $I_1, I_2 :(0,1)\to[0,\infty)$
by
$$I_1(r):=M\int_{x\in S_1(r)} |x|^{-6}\,dx\quad \text{and}\quad I_2(r):=\int_{x\in S_2 (r)} |x|^{-4}u(x)\,dx$$ where 
\[S_1(r):=\{ x\in \mathbb{R}^4 :r<|x|<1 \text{ and } 1
<u(x)\leq M|x|^{-2} \}\] 
and 
\[S_2(r):=\{ x\in \mathbb{R}^4 :r<|x|<1\text{ and }0<u(x)\leq 1 \}.\]  Then
$S_1(r) \cup S_2(r)=B_1(0)-\overline{B_r(0)}$,
\begin{equation}\label{eq5.14}
 I_2(r)=O\left(\log\frac{1}{r}\right)\quad \text{as}\quad r\to0^{+} ,
\end{equation}
and for $0<r<\frac{1}{4}$ we have
\begin{align}\label{eq5.15}
 I_1(r)+I_2(r) &\geq \int_{r<|x|<1} |x|^{-4} u(x)\,dx\\
&\geq \frac{1}{C} \int_{r<|y|<1} \left(\log \frac{e|y|}{r}\right)^2(-\Delta^2
 u(y))\,dy -C\left(\log\frac{e}{r}\right)^2
\label{eq5.16}\end{align}
by \eqref{eq5.13}.

By \eqref{eq3.3} we have
\begin{equation}\label{eq5.17}
 \int_{r<|y|<1} \left(\log\frac{e|y|}{r}\right)^2(-\Delta^2 u(y))\,dy\geq J_1(r)+J_2(r)\quad
\text{for}\quad 0<r<1
\end{equation}
where 
\[J_2(r):=\int_{S_2(r)}|y|^{\sigma-8}u(y)^{3-\sigma/2}\,dy\] 
and

\[
J_1(r):= \int_{S_1(r)} \left(\log \frac{e|y|}{r}\right)^2|y|^{\sigma-8} f(u(y))\,dy.
\]

Before continuing with the proof of Theorem \ref{thm1.2}, we prove the
following lemma.

\begin{lem}\label{lem5.1}
As $r\to 0^+$ we have
\begin{equation}\label{eq5.17.5}
J_1(r)=O\left(\left(\log\frac{1}{r}\right)^2\right),
\end{equation}
\begin{equation}\label{eq5.18}
J_2(r)=O\left(\left(\log\frac{1}{r}\right)^2\right),
\end{equation}
and
\begin{equation}\label{eq5.19}
I_1(r)=o\left(\left(\log\frac{1}{r}\right)^2\right).
\end{equation}
\end{lem}

\begin{proof}
By \eqref{eq5.17}, \eqref{eq5.16}, and \eqref{eq5.14} we 
have  
\begin{align}\notag
 J_1(r) &\leq J_1(r)+J_2(r) \leq C \left[ I_1(r)+I_2(r)
+\left(\log\frac{e}{r}\right)^2 \right]\\
 \label{eq5.19.5}
&=CI_1(r)+O\left(\left(\log\frac{1}{r}\right)^2\right)\quad\text{ as }\quad r\to 0^+.
\end{align}
If $S_1(0)=\emptyset$ then $I_1(r)\equiv J_1(r)\equiv 0$ for $0<r<1$
and thus \eqref{eq5.18} follows from \eqref{eq5.19.5}. Hence we can
assume $S_1(0)\not=\emptyset$. So for $r$ small and positive,
$S_1(r)\not=\emptyset$, $I_1(r)>0$, and 
\begin{align*}
 J_1(r)&\geq \left[ \inf_{y\in S_1(r)} \left(\log
   \frac{e|y|}{r}\right)^2|y|^{\sigma-2}f(u(y)) \right] 
\frac{I_1(r)}{M}\\
 &\geq \left[ \inf_{r<|y|<1}\left (\log\frac{e|y|}{r}\right)^2
(|y|^{-2})^{1-\sigma/2}f(M|y|^{-2})
  \right] 
\frac{I_1(r)}{M}
\end{align*}
by \eqref{eq3.2} and because $1< u(y)\leq M|y|^{-2}$ for
$y\in S_1(r)$.  Thus 
\begin{equation*}
 \frac{M^{2-\sigma/2} J_1 (r)}{I_1(r)} \geq \min \left\{\ \inf_{r<|y|<\sqrt{r}}
 (M|y|^{-2})^{1-\sigma/2}f(M|y|^{-2}),\ \left(\log\frac{e}{\sqrt{r}}\right)^2 
\inf_{\sqrt{r}<|y|<1} (M|y|^{-2})^{1-\sigma/2}f(M|y|^{-2}) \right\}\
 \to \infty 
\end{equation*}
as $r\to 0^+$ by \eqref{eq1.4}.  
Hence \eqref{eq5.19.5} implies \eqref{eq5.17.5}; and \eqref{eq5.17.5}
implies \eqref{eq5.19}. Finally, \eqref{eq5.18} follows from
\eqref{eq5.19.5} and \eqref{eq5.19}.
\end{proof}

Continuing with the proof of Theorem \ref{thm1.2}, it follows from
\eqref{eq5.14}, \eqref{eq5.19}, and \eqref{eq5.16} that there is a
constant $C>0$ such that for $0<r<1$ we have
\begin{align*}
 C&\geq \frac{1}{\left(\log\frac{e}{r}\right)^2} 
\int_{r<|y|<1} \left(\log\frac{e|y|}{r}\right)^2(-\Delta^2 u(y))\,dy\\
 &\geq \frac{1}{\left(\log\frac{e}{r}\right)^2} 
\int_{\sqrt{r}<|y|<1} \left(\log\frac{e}{\sqrt{r}}\right)^2(-\Delta^2 u(y))\,dy\\
 &\geq \frac{1}{4} \int_{\sqrt{r}<|y|<1} -\Delta^2 u(y)\,dy.
\end{align*}
Thus
\begin{equation}\label{eq5.20}
 \int_{|y|<1} -\Delta^2 u(y)\,dy<\infty.
\end{equation}

By \eqref{eq5.19} there exists a constant $C>0$ such that
$I_1(2^{-(j+1)}) \leq C(j+1)^2$ for $j=0,1,2,\ldots$.  Thus for each
$\vp>0$ we have
\begin{align*}
 M \int_{S_1(0)}|x|^{-6+\vp}\,dx&\leq
 \sum^{\infty}_{j=0}2^{-\vp j}M 
\int_{\substack{2^{-(j+1)}<|x|<2^{-j}\\ x\in S_1(0)}} |x|^{-6} \,dx\\
 &\leq \sum^{\infty}_{j=0} 2^{-\vp j} I_1 (2^{-(j+1)})\\
 &\leq C\sum^{\infty}_{j=0}(j+1)^2(2^{-\vp})^j <\infty.
\end{align*}
Hence, for $0<\vp \leq 1$,
\[
 \int_{|x|<1} |x|^{-4+\vp} u(x)\,dx 
\leq M \int_{S_1(0)} |x|^{-6+\vp} \,dx+\int_{|x|<1} |x|^{-4+\vp} \,dx
 <\infty
\]
and so taking $\vp=1$ we have for $0<r<1$ that
\begin{equation}\label{eq5.21}
 \int_{|x|<r} u(x)\,dx \leq \int_{|x|<r} \frac{r^3}{|x|^3} u(x)\,dx=o(r^3)
 \quad\text{as}\quad r\to 0^+.
\end{equation}

Let $F=\Delta^2 u$ and 
$$\hat{N}(x)=\int_{B_1(0)} 
\Phi(x-y)F(y)\,dy\quad \text{for}\quad x\in \mathbb{R}^4 .$$  
By \eqref{eq5.0} and \eqref{eq5.20} we have for $0<r\le 1$ that
\begin{align}
 \notag \int_{|x|<r} |\hat{N}(x)|\,dx 
 &= \int_{|y|<1} \left( \int_{|x|<r} A\log\frac{e}{|x-y|} \,dx \right) (-F(y))\,dy\\
 \notag &\leq \int_{|y|<1} \left( \int_{|x|<r} A\log\frac{e}{|x|} \,dx \right) (-F(y))\,dy\\
 &\le Cr^4\log\frac{e}{r} .\label{eq5.22}
\end{align}
In particular $\hat{N} \in L^{1}_{\text{loc}} (B_1(0))$.  
Also for $\varphi \in C^{\infty}_{0} (B_1(0))$ we have
\begin{align*}
 \int_{B_1(0)} \hat{N} \Delta^2 \varphi \,dx &= \int_{B_1(0)} \left( \int_{B_1(0)} \Phi(x-y)\Delta^2 \varphi (x)\,dx \right) F(y)\,dy\\
 &=\int_{B_1(0)} \varphi (y)F(y)\,dy.
\end{align*}
Thus $\Delta^2 \hat{N}=F$ in ${\cal D\,}'(B_1(0))$.

Let $v=u-\hat{N}$.  By \eqref{eq5.21} and \eqref{eq5.22} we see that
$v$ satifies \eqref{eq2.3}.
Since $\Delta^2 u=F$ in ${\cal D\,}' (B_1(0) \setminus \{0\})$ we have
$\Delta^2 v=\Delta^2 u-\Delta^2 \hat{N}=0$ in ${\cal D\,}' (B_1(0) \setminus
\{0\})$.  
Thus Lemma \ref{lem2.5} implies for some constant $a$ and some
$C^\infty$ solution $H$ of $\Delta^2
H=0$ in $B_1(0)$ we have
\begin{equation}\label{eq5.24}
 u=\hat{N}+a\Phi+H \quad \text{in}\quad B_1 (0)\setminus \{0\}.
\end{equation}
Hence, since $\hat{N}<0$ and $u>0$ we have $a\geq 0$.
\medskip

\noindent \underline{Case I.} Suppose the constant $a$ in
\eqref{eq5.24} is positive.  By \eqref{eq5.0} we have
\begin{equation*}
 \frac{\rho^{-4}}{A|\partial B_1 |} \int_{|x|=\rho} \Phi(x-y)\,dS_x \leq 
 \begin{cases}
  \rho^{-1}\log\frac{e}{\rho}, &\text{if }|y|<\rho\\
   \rho^{-1}\log\frac{e}{|y|}, &\text{if }|y|>\rho.
 \end{cases}
\end{equation*}
Thus for $0<r<1$ we have
\begin{align*}
 \frac{1}{A|\partial B_1 |} \int_{r<|x|<1} |x|^{-4} &(-\hat{N}(x))\,dx 
 = \int_{|y|<1} \int^{1}_{r}\frac{\rho^{-4}}{A|\partial B_1 |} 
 \int_{|x|=\rho} \Phi(x-y)\,dS_x \,d\rho(-\Delta^2 u(y))\,dy\\
&\leq \int_{r<|y|<1} \left( \int^{|y|}_{r} \rho^{-1}
 \log\frac{e}{|y|}\,d\rho
 +2\int^{1}_{|y|} \rho^{-1}\log\frac{e}{\rho}\,d\rho \right) (-\Delta^2 u(y))\,dy\\
&+ \int_{|y|<r} \left( \int^{1}_{r} 
 \rho^{-1}\log\frac{e}{\rho} \,d\rho \right)(-\Delta^2 u(y))\,dy\\
&\leq\left(\log\frac{e}{r}\right) \int_{r<|y|<1} \left(\log\frac{e}{|y|}\right)(-\Delta^2 u(y))\,dy
 +\left(\log\frac{e}{r}\right)^2 \int_{|y|<r} -\Delta^2 u(y)\,dy\\
&=o\left(\left(\log\frac{e}{r}\right)^2\right)\quad \text{as}\quad r\to 0^{+}
\end{align*}
by \eqref{eq5.20} and the fact that
\begin{align*}
 \int_{r<|y|<1} \left(\log\frac{e}{|y|}\right)(-\Delta^2 u(y))\,dy 
&\leq\left(\log\frac{e}{r}\right) \int_{r<|y|<(\log\frac{1}{r})^{-1}} -\Delta^2 u(y)\,dy\\
 &+\left(\log\left(e \log\frac{1}{r}\right)\right) \int_{|y|<1} -\Delta^2 u(y)\,dy.
\end{align*}
Hence by \eqref{eq5.24}
\begin{align*}
 \int_{r<|x|<1} |x|^{-4} u(x)\,dx &=a \int_{r<|x|<1} 
|x|^{-4}A\log\frac{e}{|x|}\,dx-o\left(\left(\log\frac{1}{r}\right)^2\right)\\
 &\geq C\left(\log\frac{1}{r}\right)^2 \quad \text{for small }r>0
\end{align*}
where $C$ is a positive constant.  Thus by \eqref{eq5.15}
and \eqref{eq5.19} we have
\begin{equation}\label{eq5.25}
 I_2(r)\geq C\left(\log\frac{1}{r}\right)^2\quad \text{for small }r>0.
\end{equation}

On the other hand, by H\"older's inequality and \eqref{eq5.18},
\begin{align*}
 I_2(r)&=\int_{S_2(r)} |x|^{\frac{2\sigma-8}{6-\sigma}}(|x|^{\frac{2\sigma-16}{6-\sigma}} u(x))\,dx\\
 &\leq \left( \int_{0<|x|<1} |x|^{-2} \,dx
 \right)^{\frac{4-\sigma}{6-\sigma}} 
J_2 (r)^{\frac{2}{6-\sigma}}\\
 &=O\left(\left(\log\frac{1}{r}\right)^{\frac{4}{6-\sigma}}\right)\quad\text{as}\quad r\to 0^{+}
\end{align*}
which contradicts \eqref{eq5.25} and completes the proof of Theorem
\ref{thm1.2} in Case I.
\medskip

\noindent \underline{Case II.} Suppose the constant $a$ in
\eqref{eq5.24} is zero.  Then
\begin{equation}\label{eq5.26}
 0<u=\hat{N}+H \quad \text{for}\quad 0<|x|<1.
\end{equation}
Thus $-\hat{N}$ and $u$ are positive and bounded for $0<|x|\le 1/2$
and so \eqref{eq3.3} implies
\begin{equation}\label{eq5.26.5}
-\Delta^2 u\ge C|x|^{\sigma-8} u^{3-\sigma/2} \quad\text{in}
\quad \overline{B_1(0)}\setminus\{0\}
\end{equation}
for some positive constant $C$. 
Also, since \eqref{eq5.0} implies
\begin{align}\label{eq5.27}
 -\frac{1}{A}\bar{\hat{N}}(r): &=\frac{1}{A|\partial B_r|} 
\int_{|x|=r} -\hat{N} (x)\,dS_x =\int_{|y|<1} \frac{1}{|\partial B_r|} 
\int_{|x|=r} 
\log\frac{e}{|x-y|}\,dS_x\,
(-F(y))\,dy\\
 \notag &\geq \frac{3}{4} \int_{r<|y|<1} \left(\log\frac{e}{|y|}\right)
 (-F(y))\,dy\quad\text{for} \quad 0<r<1/2
\end{align}
we see that
\begin{equation}\label{eq5.28}
 -\hat{N}_0 :=\int_{|y|<1} A\left(\log\frac{e}{|y|}\right) (-F(y))\,dy<\infty.
\end{equation}
Averaging \eqref{eq5.26} we obtain
\begin{equation}\label{eq5.29}
 0<\bar{u}(r)=\overline{(\hat{N}-\hat{N}_0 )}(r)+a_0 -a_1 Ar^2 
\quad \text{for}\quad 0<r<1
\end{equation}
for some constants $a_0$ and $a_1$.

By \eqref{eq5.27}, \eqref{eq5.28} and \eqref{eq5.0} we have
\begin{align}
 \notag &\frac{1}{A}\overline{(\hat{N}-\hat{N}_0 )}(r)\\
 \notag &= \int_{|y|<1} \left(\frac{1}{|\partial B_r |} 
\int_{|x|=r}  
\log\frac{e}{|x-y|}\,dS_x-\log\frac{e}{|y|} \right) F(y)\,dy\\
 \label{eq5.30} &=\int_{|y|<r} 
\left(\log\frac{e}{r}-\log\frac{e}{|y|}-\frac{1}{4}\left(\frac{|y|}{r}\right)^2\right)
F(y)\,dy
-\int_{r<|y|<1} 
\frac{1}{4}r^2 |y|^{-2} 
F(y)\,dy\\
 \notag &=-\int_{|y|<r} 
\left(\log\frac{e}{|y|}\right)
\left[1-\frac{\log\frac{e}{r}}{\log\frac{e}{|y|}}
+\frac{1}{4}\frac{\left(\frac{|y|}{r}\right)^2}{\log\frac{e}{|y|}}\right]
F(y)\,dy\\
 \notag &-\int_{r<|y|<\sqrt{r}} 
\frac{1}{4}\frac{\left(\frac{r}{|y|}\right)^2}{\log\frac{e}{|y|}}
\left(\log\frac{e}{|y|}\right)
F(y)\,dy
-\int_{\sqrt{r}<|y|<1} 
\frac{1}{4}\frac{\left(\frac{r}{|y|}\right)^2}{\log\frac{e}{|y|}}
\left(\log\frac{e}{|y|}\right)
F(y)\,dy\\
 \notag &\to0 \quad \text{as}\quad r\to 0^{+}
\end{align}
by \eqref{eq5.28}.  Thus by \eqref{eq5.29} we have $a_0 \geq 0$.  
If $a_0>0$ then by \eqref{eq5.26.5} we have
\begin{align*}
 -\bar{F}(r)&=-\overline{\Delta^2 u}(r)\geq Cr^{\sigma-8} \bar{u}(r)^{3-\sigma/2}\\
 &\geq Cr^{\sigma-8} \left(\frac{a_0}{2}\right)^{3-\sigma/2}
\end{align*}
for $r$ small and positive and thus for small $r_0>0$ we have
\begin{align*}
 \int_{|y|<r_0} \left(\log\frac{e}{|y|}\right)(-F(y))\,dy 
&= \int^{r_0}_{0} \left(\log\frac{e}{r}\right)\int_{|y|=r} -F(y)\,dS_y \,dr\\
 &=|\partial B_1|\int^{r_0}_{0} r^3\left(\log\frac{e}{r}\right) (-\bar{F}(r))\,dr=\infty
\end{align*}
which contradicts \eqref{eq5.28}.  So $a_0=0,\bar{u}(r)\to 0$ as $r\to0$ 
and by \eqref{eq5.29} and \eqref{eq5.30} we have
\begin{align}
 \notag \frac{\bar{u}(r)}{A} 
&= \frac{\overline{(\hat{N}-\hat{N_0})}(r)}{A} -a_1 r^2\\
 \label{eq5.31}&= \int_{|y|<r}
\left(\log\frac{e}{|y|}\right)
\left[1-\frac{\log\frac{e}{r}}{\log\frac{e}{|y|}}
+\frac{1}{4}\frac{\left(\frac{|y|}{r}\right)^2}{\log\frac{e}{|y|}}\right]
(-\Delta^2 u(y))\,dy+(J(r)-a_1 )r^2
\end{align}
where
\begin{equation*}
 J(r):= \int_{r<|y|<1} \frac{1}{4}|y|^{-2} (-\Delta^2 u(y))\,dy 
\quad \text{for}\quad 0\leq r<1.
\end{equation*}
($J(0)$ may be $\infty$.)
\medskip

\noindent \underline{Case II(a).} Suppose $a_1 <J(0)$.  Then there
exists $\vp>0$ and $r_0 \in (0,1)$ such that $a_1 \leq
(1-\vp)J(r)$ for $0<r\leq r_0$.  Thus by \eqref{eq5.31},
$\frac{\bar{u}(r)}{A} \geq \vp r^2 J(r)$ for $0<r<r_0$, which together
with \eqref{eq5.26.5} and Lemma \ref{lem2.6} gives a contradiction and
thereby proves Theorem \ref{thm1.2} when $n=4$ in Case II(a).
\medskip

\noindent\underline{Case II(b).} Suppose $a_1 \geq J(0)$.  Then for
$0<r<1$ we have
\begin{equation*}
 (J(r)-a_1)r^2 \leq (J(r)-J(0))r^2 =-\int_{|y|<r} 
\frac{1}{4} r^2 |y|^{-2} (-\Delta^2 u(y))\,dy
\end{equation*}
and hence by \eqref{eq5.31} we have
\begin{align}
 \notag 0<\frac{\bar{u}(r)}{A} &\leq \int_{|y|<r} 
\left(\log\frac{r}{|y|}+\frac{1}{4}\left(\frac{|y|}{r}\right)^2
-\frac{1}{4}\left(\frac{r}{|y|}\right)^2\right)
(-\Delta^2 u(y))\,dy\\
 \label{eq5.32}&=-\int_{|y|<r} |y|^2 r^{-2}
 p\left(\frac{r}{|y|}\right)(-\Delta^2 u(y))\,dy
\end{align}
where $p(t):=\frac{1}{4}t^4-t^2\log t-\frac{1}{4}$.  
Since $p(1)=p'(1)=p''(1)=0$ and
\[
 p'''(t) =6t^{-1}\left(t^2-\frac{1}{3}\right)>0
\quad\text{for}\quad t\geq 1
\]
we see that $p(t)>0$ for $t>1$.  This contradicts \eqref{eq5.32} 
and completes the proof of Theorem \ref{thm1.2} when $n=4$ in
all cases.
\end{proof}

\section{Completion of the Proof of Theorem \ref{thm1.2} when $n= 3$} 
\label{sec6}

When $n=3$, we complete in this section the proof of Theorem
\ref{thm1.2} which we began in Section \ref{sec3}.

\begin{proof}[Completion of the proof of Theorem \ref{thm1.2} when
  $n=3$] 
For $x\in \mathbb R^3$ we see by Lemma \ref{lem2.1} that
\begin{equation}\label{eq6.0}
\frac{1}{|\partial B_r |} \int_{|y|=r} |x-y|\,dS_y=
\begin{cases}
   r+\frac{1}{3} r^{-1} |x|^2 , &\text{if } |x|<r \\
   |x|+\frac{1}{3}r^2 |x|^{-1}, &\text{if } |x|>r.
  \end{cases}
 \end{equation}
It therefore follows from equations \eqref{eq1.12} and \eqref{eq3.5} 
that for $r>0$ we have
\begin{equation}\label{eq6.10.5}
 \frac{-1}{A|\partial B_r |} \int_{|x|=r} \Psi (x,y)\,dS_x=
 \begin{cases}
  \frac{1}{3}r^{-1} |y|^2 , &\text{if }|y|<r \\
  |y| p\left(\frac{r}{|y|}\right), &\text{if }|y|\geq r
 \end{cases}
\end{equation}
where $p(t):=1-t+\frac{1}{3}t^2$ is bounded between
positive constants for $0\leq t \leq 1$.  Hence
\[
 -\int_{|x|=r} \Psi(x,y)\,dS_x \sim 
 \begin{cases}
  2r|y|^2, &\text{if }|y|<r\\
  r^2|y|, &\text{if }|y|\geq r
 \end{cases}
\quad\text{ for }\quad (r,y) \in (0,\infty) \times \mathbb{R}^3 .
\]
Thus by \eqref{eq3.9}, for $0<r\leq
\frac{1}{4}$, we have
\begin{align}
 \notag \int_{r<|x|<1} |x|^{-4} N(x)\,dx =& \int_{|y|<1} \int^{1}_{r} 
\rho^{-4} \int_{|x|=\rho} -\Psi(x,y)\,dS_x \,d\rho (-\Delta^2 u(y))\,dy\\
\notag \sim& \int_{r<|y|<1} \left( \int^{1}_{|y|} 2\rho^{-3} |y|^2\,d\rho +
\int^{|y|}_{r} \rho^{-2}|y| \,d\rho \right) (-\Delta^2 u(y))\,dy\\
\notag &+ \int_{|y|<r}\left(\int_r^12\rho^{-3}|y|^2\,d\rho\right)(-\Delta^2 u(y))\,dy\\
\sim& \int_{r<|y|<1} \left(\frac{|y|}{r}\right)(-\Delta^2
u(y))\,dy+g(r)
\label{eq6.11}
\end{align}
where
\begin{equation}\label{eq6.12}
 0<g(r):=\left(\frac{1}{r^2}-1\right)\int_{|y|<r} |y|^2 (-\Delta^2
 u(y))\,dy=o\left(\frac{1}{r^2}\right) \quad
\text{as}\quad r\to0^{+}
\end{equation}
by \eqref{eq3.7}.

For $0<r<1$ we have
\begin{align*}
 \int_{r<|y|<1} \left(\frac{|y|}{r}\right)&(-\Delta^2
  u(y))\,dy=\int_{r<|y|<1} \frac{1}{r|y|}|y|^2 (-\Delta^2 u(y))\,dy\\
 &\leq \frac{1}{r^2}\int_{r<|y|<\sqrt{r}} |y|^2(-\Delta^2 u(y))\,dy+
\frac{1}{r^{3/2}}\int_{\sqrt{r}<|y|<1} |y|^2 (-\Delta^2 u(y))\,dy\\
 &\leq \frac{1}{r^2} \int_{r<|y|<\sqrt{r}} |y|^2(-\Delta^2 u(y))\,dy
+\frac{1}{r^{3/2}}\int_{|y|<1} |y|^2 (-\Delta^2 u(y))\,dy\\
 &=o(r^{-2})\quad \text{as}\quad r\to 0^{+}
\end{align*}
by \eqref{eq3.7}. It follows therefore from \eqref{eq6.11} and
\eqref{eq6.12} that
$$\int_{r<|x|<1} |x|^{-4}N(x)\,dx=o(r^{-2})\quad \text{as}\quad r\to0^{+}.$$ 
Hence, by \eqref{eq3.10} and the positivity of $u$ we see
that the constant $c_4$ in \eqref{eq3.10} is nonnegative and thus by
\eqref{eq3.10}, \eqref{eq6.11}, and the positivity of $g$ we have
\begin{equation}\label{eq6.13}
 \int_{r<|x|<1} |x|^{-4} u(x)\,dx \geq \frac{1}{C} \int_{r<|y|<1} \left(
 \frac{|y|}{r}\right)(-\Delta^2 u(y))\,dy-C\frac{1}{r}\quad \text{for}
\quad 0<r<\frac{1}{4}
\end{equation}
where $C$ is a positive constant independent of $r$.  

By \eqref{eq3.6} there exists a constant $M>1$ such that $0<u(x)\leq
M|x|^{-1}$ for $0<|x|\leq 1$. Define $I_1, I_2 :(0,1)\to[0,\infty)$
by
$$I_1(r):=M\int_{x\in S_1(r)} |x|^{-5}\,dx\quad \text{and}\quad I_2(r):=\int_{x\in S_2 (r)} |x|^{-4}u(x)\,dx$$ where 
\[S_1(r):=\{ x\in \mathbb{R}^3 :r<|x|<1 \text{ and }|x|
<u(x)\leq M|x|^{-1} \}\] 
and 
\[S_2(r):=\{ x\in \mathbb{R}^3 :r<|x|<1\text{ and }0<u(x)\leq |x| \}.\]  Then
$S_1(r) \cup S_2(r)=B_1(0)-\overline{B_r(0)}$,
\begin{equation}\label{eq6.14}
 I_2(r)=O\left(\log\frac{1}{r}\right)\quad \text{as}\quad r\to0^{+} ,
\end{equation}
and for $0<r<\frac{1}{4}$ we have
\begin{align}\label{eq6.15}
 I_1(r)+I_2(r) &\geq \int_{r<|x|<1} |x|^{-4} u(x)\,dx\\
&\geq \frac{1}{C} \int_{r<|y|<1} \left(\frac{|y|}{r}\right)(-\Delta^2
u(y))\,dy-C\frac{1}{r}
\label{eq6.16}\end{align}
by \eqref{eq6.13}.

By \eqref{eq3.3} we have
\begin{equation}\label{eq6.17}
 \int_{r<|y|<1} \left(\frac{|y|}{r}\right)(-\Delta^2 u(y))\,dy\geq J_1(r)+J_2(r)\quad
\text{for}\quad 0<r<1
\end{equation}
where 
\[J_2(r):=\int_{S_2(r)}|y|^{2\sigma-12}u(y)^{5-\sigma}\,dy\] 
and 
\[
J_1(r):= \int_{S_1(r)} \left(\frac{|y|}{r}\right)|y|^{\sigma-7} f(|y|^{-1} u(y))\,dy.
\]

Before continuing with the proof of Theorem \ref{thm1.2}, we prove the
following lemma.

\begin{lem}\label{lem6.1}
As $r\to 0^+$ we have
\begin{equation}\label{eq6.17.5}
J_1(r)=O\left(\frac{1}{r}\right),
\end{equation}
\begin{equation}\label{eq6.18}
J_2(r)=O\left(\frac{1}{r}\right),
\end{equation}
and
\begin{equation}\label{eq6.19}
I_1(r)=o\left(\frac{1}{r}\right).
\end{equation}
\end{lem}

\begin{proof}
By \eqref{eq6.17}, \eqref{eq6.16}, and \eqref{eq6.14} we 
have  
\begin{align}\notag
 J_1(r) &\leq J_1(r)+J_2(r) \leq C \left[ I_1(r)+I_2(r)+\frac{1}{r} \right]\\
 \label{eq6.19.5}
&=CI_1(r)+O\left(\frac{1}{r}\right)\quad\text{ as }\quad r\to 0^+.
\end{align}
If $S_1(0)=\emptyset$ then $I_1(r)\equiv J_1(r)\equiv 0$ for $0<r<1$
and thus \eqref{eq6.18} follows from \eqref{eq6.19.5}. Hence we can
assume $S_1(0)\not=\emptyset$. So for $r$ small and positive,
$S_1(r)\not=\emptyset$, $I_1(r)>0$, and 
\begin{align*}
 J_1(r)&\geq \left[ \inf_{y\in S_1(r)} \left(
   \frac{|y|}{r}\right)|y|^{\sigma-2}f(|y|^{-1} u(y)) \right] 
\frac{I_1(r)}{M}\\
 &\geq \left[ \inf_{r<|y|<1}
   \left(\frac{|y|}{r}\right)(|y|^{-2})^{1-\sigma/2} 
f(M|y|^{-2})
  \right] 
\frac{I_1(r)}{M}
\end{align*}
by \eqref{eq3.2} and because $1< |y|^{-1} u(y)\leq M|y|^{-2}$ for
$y\in S_1(r)$.  Thus 
\begin{equation*}
 \frac{M^{2-\sigma/2} J_1 (r)}{I_1(r)} \geq \min \left\{\ \inf_{r<|y|<\sqrt{r}}
 (M|y|^{-2})^{1-\sigma/2}f(M|y|^{-2}),\ \left(\frac{1}{\sqrt{r}}\right) 
\inf_{\sqrt{r}<|y|<1} (M|y|^{-2})^{1-\sigma/2}f(M|y|^{-2}) \right\}\
 \to \infty 
\end{equation*}
as $r\to 0^+$ by \eqref{eq1.4}.  
Hence \eqref{eq6.19.5} implies \eqref{eq6.17.5}; and \eqref{eq6.17.5}
implies \eqref{eq6.19}. Finally, \eqref{eq6.18} follows from
\eqref{eq6.19.5} and \eqref{eq6.19}.
\end{proof}

Continuing with the proof of Theorem \ref{thm1.2}, it follows from
\eqref{eq6.14}, \eqref{eq6.19}, and \eqref{eq6.16} that 
\begin{equation}\label{eq6.20}
 \int_{|y|<1} |y|(-\Delta^2 u(y))\,dy<\infty.
\end{equation}

Since, by \eqref{eq6.20}, 
\eqref{eq3.3}, and \eqref{eq3.2}, 
\begin{align*}
\infty>\int_{|x|<1}|x|(-\Delta^2u(x))\,dx
&\ge \int_{S_1(0)} |x||x|^{\sigma-7}f(|x|^{-1}u(x))\,dx\\
&\ge \frac{1}{M^{1-\sigma/2}}\int_{S_1(0)}|x|^{-4}(M|x|^{-2})^{1-\sigma/2}f(M|x|^{-2})\,dx\\
&\ge \frac{1}{M^{1-\sigma/2}}
\left(\min_{t\ge M}t^{1-\sigma/2}f(t)\right)\int_{S_1(0)}|x|^{-4}\,dx,
\end{align*}
it follows from \eqref{eq1.4} that $\int_{S_1(0)}|x|^{-4}\,dx<\infty$. Hence
\begin{equation}\label{eq1}
\int_{|x|<1}|x|^{-3}u(x)\,dx\le M\int_{S_1(0)}|x|^{-4}\,dx +
\int_{B_1(0)\setminus S_1(0)} |x|^{-2}\,dx<\infty.
\end{equation}
Thus
by Lemma \ref{lem2.7} we have
\begin{equation}\label{eq6.21}
 \int_{|y|<1} -\Delta^2 u(y)\,dy<\infty.
\end{equation}
Equation \eqref{eq1} also implies
\begin{equation}\label{eq6.21.5}
 \int_{|x|<r} u(x)\,dx\le r^3\int_{|x|<r}|x|^{-3}u(x)\,dx=o(r^3)
 \quad\text{as}\quad r\to 0^+.
\end{equation}

Let $F=\Delta^2 u$ and 
$$\hat{N}(x)=\int_{B_1(0)} 
\Phi(x-y)F(y)\,dy\quad \text{for}\quad x\in \mathbb{R}^3 .$$  
It follows from \eqref{eq1.12} and \eqref{eq6.21} that $\hat N\in
C^1(\mathbb{R}^3)$. In particular
\begin{equation}\label{eq6.22}
U(x):=\hat N(x)-\hat N(0)-D\hat N(0)x=o(|x|) \quad\text{as $x\to 0$.}
\end{equation}
Also for $\varphi \in C^{\infty}_{0} (B_1(0))$ we have
\begin{align*}
 \int_{B_1(0)} \hat{N} \Delta^2 \varphi \,dx &= \int_{B_1(0)} \left( \int_{B_1(0)} \Phi(x-y)\Delta^2 \varphi (x)\,dx \right) F(y)\,dy\\
 &=\int_{B_1(0)} \varphi (y)F(y)\,dy.
\end{align*}
Thus $F=\Delta^2 \hat{N}=\Delta^2 U$ in ${\cal D\,}'(B_1(0))$.

Let $v=u-U$.  By \eqref{eq6.21.5} and \eqref{eq6.22} we see that $v$
satisfies \eqref{eq2.3}.
Since $\Delta^2 u=F$ in ${\cal D\,}' (B_1(0) \setminus \{0\})$ we have
$\Delta^2 v=\Delta^2 u-\Delta^2 U=0$ in ${\cal D\,}' (B_1(0) \setminus
\{0\})$.  
Thus Lemma \ref{lem2.5} implies for some constant $b_0$ and some 
$C^\infty$ solution $H$ of $\Delta^2 H=0$ in $B_1(0)$ we have 
\begin{align}\label{eq6.24}
 u&=U(x)+b_0|x|+H(x)\quad \text{in}\quad B_1 (0)\setminus \{0\}\\
  &=b_0|x|+H(x)+o(|x|) \quad\text{as $x\to 0$}\label{eq6.25}
\end{align}
by \eqref{eq6.22}. 
Hence the positivity of $u$ implies $H(0)\ge 0$. If $H(0)>0$ then by
\eqref{eq6.25} we have $u(x)>\vp>0$ for $|x|$ small and positive which
contradicts \eqref{eq6.21.5}. Thus $H(0)=0$. Hence \eqref{eq6.25}
implies 
\begin{equation}\label{eq6.29}
u(x)=O(|x|)\quad \text{as $x\to 0$}
\end{equation}
and the biharmonicity of $H$ implies
\begin{equation}\label{eq6.28}
\bar H(r)=-b_1Ar^2 \quad \text{for some constant $b_1$}
\end{equation}
where $\bar H$ is the average of $H$ on the sphere $|x|=r$.
By \eqref{eq6.29} and \eqref{eq3.3} we see for some positive constant
$C$ that
\begin{equation}\label{eq6.29.1}
-\Delta^2u(x)\ge C|x|^{2\sigma-12}u(x)^{5-\sigma} \quad\text{for}\quad 0<|x|\le 1.
\end{equation}
Averaging \eqref{eq6.25} we find by \eqref{eq6.28} that
\[
0<\bar u(r)=b_0r+o(r) \quad\text{as $r\to 0$.}
\]
Hence $b_0 \ge 0$.  If $b_0>0$ then 
for some constant $C>0$ we have $\bar u(r)\ge Cr$ for $0<r\le 1$ and
thus averaging \eqref{eq6.29.1} we get $-\overline{\Delta^2u}(r)\ge
Cr^{\sigma-7}\ge Cr^{-5}$ for $0<r\le 1$ which contradicts \eqref{eq6.21}. Hence $b_0=0$
and so averaging \eqref{eq6.24} and using \eqref{eq6.28} we get
\begin{equation}\label{eq6.29.5}
\frac{\bar u(r)}{A}=\frac{\bar U(r)}{A}-b_1r^2\quad \text{for $0<r\le 1$.}
\end{equation}
It follows from \eqref{eq6.22} and \eqref{eq6.0} that for $0<r<1$ we have
\begin{align}
\notag\frac{1}{A}\bar U(r)&=\frac{1}{A}\overline{(\hat{N}-\hat{N}(0))}(r)\\
 \notag &= \int_{|y|<1} \left( \frac{1}{|\partial B_r |} 
 \int_{|x|=r} |x-y|\,dS_x -|y| \right) (-\Delta^2u(y))\,dy\\ 
\notag 
 &=\int_{|y|<r} \left(r -|y| +
 \frac{1}{3} r^{-1}|y|^2\right)(-\Delta^2u(y))\,dy
+\int_{r<|y|<1} \frac{1}{3}r^2 |y|^{-1} (-\Delta^2u(y))\,dy.
\end{align} 
Hence by \eqref{eq6.29.5},
\begin{equation}\label{eq6.31}
 \frac{\bar{u}(r)}{A} 
= \int_{|y|<r} r\left[1-\frac{|y|}{r} 
+\frac{1}{3} \left(\frac{|y|}{r}\right)^2 \right] 
(-\Delta^2 u(y))\,dy+(J(r)-b_1 )r^2
\end{equation}
where
\begin{equation*}
 J(r):= \int_{r<|y|<1} \frac{1}{3} |y|^{-1} (-\Delta^2 u(y))\,dy 
\quad \text{for}\quad 0\le r< 1.
\end{equation*}
($J(0)$ may be $\infty$.)
\medskip

\noindent \underline{Case I.} Suppose $b_1 <J(0)$.  Then there
exists $\vp>0$ and $r_0 \in (0,1)$ such that $b_1 \leq
(1-\vp)J(r)$ for $0<r\leq r_0$.  Thus by \eqref{eq6.31},
$\frac{\bar{u}(r)}{A} \geq \vp r^2 J(r)$ for $0<r<r_0$, which together
with \eqref{eq6.29.1} and Lemma \ref{lem2.6} gives a contradiction and
thereby proves Theorem \ref{thm1.2} when $n=3$ in Case I.
\medskip

\noindent\underline{Case II.} Suppose $b_1 \geq J(0)$.  Then for
$0<r<1$ we have
\begin{equation*}
 (J(r)-b_1)r^2 \leq (J(r)-J(0))r^2 =-\int_{|y|<r} 
\frac{1}{3} r^2 |y|^{-1} (-\Delta^2 u(y))\,dy
\end{equation*}
and hence by \eqref{eq6.31} we have
\begin{align*}
 0<\frac{\bar{u}(r)}{A} 
&\leq \int_{|y|<r} r\left[ 1-\frac{|y|}{r} 
+\frac{1}{3} \left(\left(\frac{|y|}{r}\right)^2 
-\frac{r}{|y|}\right)\right] (-\Delta^2 u(y))\,dy\\
 \label{eq6.32}
&=-\frac{1}{3}\int_{|y|<r} |y|^2 r^{-1} 
\left(\frac{r}{|y|}-1\right)^3(-\Delta^2 u(y))\,dy.
\end{align*}
 This contradiction 
completes the proof of Theorem \ref{thm1.2} in
all cases.
\end{proof}

\section{Proof of Theorem \ref{thm1.3}} \label{sec7}
In this section we prove Theorem \ref{thm1.3}.

\begin{proof}[Proof of Theorem \ref{thm1.3}.]
Suppose for contradiction
that $v(y)$ is a $C^4$ positive solution of \eqref{eq1.3} in
$\mathbb{R}^2\setminus\Omega$.  By scaling $v$, we can assume
$\Omega=\overline{B_{1/2}(0)}$. 
Let $u(x)=|y|^{-2} v(y)$, $y=\frac{x}{|x|^2}$ be the $2$-Kelvin
transform of $v(y)$.  Then
\begin{equation}\label{eq7.1}
v(y)=|x|^{-2} u(x)\quad\text{and} \quad \Delta^2 v(y)=|x|^6 \Delta^2 u(x). 
\end{equation}
It follows therefore from \eqref{eq1.3}
that $u(x)$ is a $C^4$ positive solution of
\begin{equation}\label{eq7.2}
 -\Delta^2 u(x)\geq |x|^{\sigma-6} f(|x|^{-2}u(x)) 
\quad\text{in }B_2 (0) \setminus \{0\}.
\end{equation}
Choose $s_0>1$ and so large that the function
\[
g(s):=
\begin{cases}\frac{(\log s)^{1-\sigma/2}}{s^{1-\sigma/2}\prod_{i=2}^k\log^is}&\text{if $s\ge s_0$}\\
g(s_0)&\text{if $0<s<s_0$}
\end{cases}
\]
is well defined, continuous, positive, and nonincreasing for $s>0$.
By \eqref{eq1.6} we have for some positive constant $C$ that 
\begin{equation}\label{eq7.5}
 f(s)\geq Cg(s) \quad\text{for}\quad s\ge 1.
\end{equation}
Since $u$ is a $C^4$ positive solution of \eqref{eqA.1} 
it follows from Theorem \ref{thmA} that $u$ satisfies \eqref{eq3.7}
and, for some constant $M>e$,
\begin{equation}\label{eq7.4}
 u(x)\le M\log\frac{e}{|x|}\quad\text{for}\quad 0<|x|\le 1.
\end{equation}
Since
\[
\frac{g(M|x|^{-2}\log\frac{e}{|x|})}{|x|^{2-\sigma}/\prod_{i=2}^k\log^i\frac1{|x|}}
\to\left(\frac{2}{M}\right)^{1-\sigma/2} \quad \text{as} \quad x\to 0
\]
there exists $r_0\in(0,1/s_0)$ such that 
\begin{equation}\label{eq7.4.5}
g\left(M|x|^{-2}\log\frac{e}{|x|}\right)
\ge \frac{|x|^{2-\sigma}}{M^{1-\sigma/2}\prod_{i=2}^k\log^i\frac1{|x|}} 
\quad \text{for $0<|x|<r_0$.} 
\end{equation}
Let $D=\{ x\in B_{r_0} (0) \setminus \{0\}: |x|^2 \leq u(x) \leq
M\log \frac{e}{|x|} \}$.  Since
$1\leq |x|^{-2} u(x)\leq M|x|^{-2} \log \frac{e}{|x|}$ for $x\in
D$, it follows from \eqref{eq7.2}, \eqref{eq7.5}, and \eqref{eq7.4.5} 
that for $\alpha>0$ we have
 \begin{align*}
  I(\alpha)&:=\int_D|x|^\alpha(-\Delta^2u(x))\,dx\\
  &\geq \int_D |x|^\alpha |x|^{\sigma-6} f(|x|^{-2} u(x)) \,dx\\
  &\geq C \int_D |x|^{\alpha+\sigma-6} g \left(M|x|^{-2} \log \frac{e}{|x|} \right) dx\\
  &\geq C \int_D \frac{|x|^{\alpha-4}}{\prod_{i=2}^k\log^i\frac{1}{|x|}}\,dx.
 \end{align*}
Hence by \eqref{eq7.4}
\begin{align}\label{eq7.6}\notag 
 \int_{|x|<r_0}
\frac{|x|^{\alpha-4}}{\log\frac{e}{|x|}\prod_{i=2}^k\log^i\frac{1}{|x|}} u(x)\,dx 
&\leq M \int_D
\frac{|x|^{\alpha-4}}{\prod_{i=2}^k\log^i\frac{1}{|x|}}\,dx
+ \int_{B_{r_0}(0) \setminus D} 
\frac{|x|^{\alpha-2}}{\log \frac{e}{|x|}\prod_{i=2}^k\log^i\frac{1}{|x|}} \,dx\\
&\le CI(\alpha)+\int_{B_{r_0}(0)}\frac{|x|^{\alpha-2}}{\log \frac{e}{|x|}\prod_{i=2}^k\log^i\frac{1}{|x|}} \,dx.
\end{align}
By \eqref{eq3.7}, $I(2)<\infty$. Thus \eqref{eq7.6} with $\alpha=2$
and Lemma \ref{lem2.7} imply
\begin{equation}\label{eq7.7}
 \int_{|x|<1} -\Delta^2 u(x)\,dx<\infty.
\end{equation}
Hence $I(1/2)<\infty$ and thus by \eqref{eq7.6} with $\alpha=1/2$ we
have 
$$\frac{1}{r^{7/2} \log \frac{e}{r}\prod_{i=2}^k\log^i\frac{1}{r}} 
\int_{|x|<r} u(x)\,dx=o(1) \quad\text{as}\quad r\to 0^{+}.$$ Therefore
\begin{equation}\label{eq7.8}
 \int_{|x|<r} u(x)\,dx= o(r^3) \quad\text{as}\quad r\to 0^{+}.
\end{equation}

Let $F=\Delta^2 u$ and
\[
N(x)=\int_{|y|<1} \Phi(x-y)F(y)\,dy\quad \text{for}\quad 
x\in \mathbb{R}^2
\]  
where $\Phi$ is the fundamental solution of $\Delta^2$ in $\mathbb
R^2$ given by \eqref{eq1.13}.
It follows from \eqref{eq7.7} and \eqref{eq1.13} that $N\in C^1
(\mathbb{R}^2)$.  In particular
\begin{equation}\label{eq7.10}
 U(x):=N(x)-N(0)-(DN)(0)x=o(|x|) \quad\text{as}\quad x\to 0.
\end{equation}
Also for $\varphi \in C^{\infty}_{0} (B_1(0))$ we have
\begin{align*}
 \int_{B_1(0)} N \Delta^2 \varphi \,dx 
&= \int_{B_1(0)} \left( \int_{B_1(0)} 
\Phi(x-y)\Delta^2 \varphi (x)\,dx \right) F(y)\,dy\\
 &=\int_{B_1(0)} \varphi (y)F(y)\,dy.
\end{align*}
Thus $F=\Delta^2N=\Delta^2U$ in ${\cal D\,}'(B_1(0))$.

Let $v=u-U$.  By \eqref{eq7.8} and \eqref{eq7.10} we see that $v$
satifies \eqref{eq2.3}.
Since $\Delta^2 u=F$ in ${\cal D\,}' (B_1(0) \setminus \{0\})$ we have
$\Delta^2 v=\Delta^2 u-\Delta^2 U=0$ in ${\cal D\,}' (B_1(0) \setminus
\{0\})$.  
Thus Lemma \ref{lem2.5} implies for some constant $b$ that
\begin{equation}\label{eq7.12}
 u(x)=U(x)+b|x|^2\log\frac{e}{|x|}+H(x) \quad
\text{in}\quad B_1 (0) \setminus \{0\}
\end{equation}
where $H$ is a $C^{\infty}$ biharmonic function in $B_1 (0)$.

 It follows from \eqref{eq7.12}, \eqref{eq7.10}, and the positivity of
 $u$ that $H(0)\geq 0$.  If
$H(0)>0$ then, by \eqref{eq7.12} and \eqref{eq7.10}, $u(x)>\vp >0$
for $|x|$ small and positive, which contradicts \eqref{eq7.8}.  Thus
$H(0)=0$. 
Hence by \eqref{eq7.12}, \eqref{eq7.10}, and the positivity
of $u$ we have $DH(0)=0$ and thus 
\begin{equation}\label{eq7.19}
 H(x)=O(|x|^2)\quad\text{as}\quad x\to 0.
\end{equation}

For $x\in \mathbb R^2$ we see by Lemma \ref{lem2.1} that
\[
\frac{1}{|\partial B_r |} \int_{|y|=r} |x-y|^2\log\frac{e}{|x-y|}\,dS_y=
\begin{cases}
   r^2\log\frac{e}{r} + |x|^2\log\frac{1}{r}, &\text{if } |x|<r \\
   |x|^2\log\frac{e}{|x|} + r^2\log\frac{1}{|x|}, &\text{if } |x|>r.
  \end{cases}
 \]
It therefore follows from \eqref{eq7.10} and \eqref{eq7.7} that for
 $0<r<e^{-1}$ we have
 \begin{align*}
 \Big|\frac{1}{A} \bar U(r) \Big| 
=&  \Big|\frac{1}{A} \overline{(N-N(0))}(r) \Big|\\ 
=& \Big| \int_{|y|<1} \frac{1}{|\partial B_r |} 
\int_{|x|=r} \left(\frac{-\Phi (x-y)}{A} +\frac{\Phi(y)}{A} \right)dS_x 
(-\Delta^2 u(y)) \,dy \Big| \\
  =& \Big| \int_{|y|<1} \left(\frac{1}{|\partial B_r |} \int_{|x|=r} 
|x-y|^2 \log \frac{e}{|x-y|} \, dS_x - |y|^2 \log \frac{e}{|y|} \right) (-\Delta^2 u(y))\,dy \Big| \\
  =& \Big| \int_{|y|<r} \left(r^2 \log \frac{e}{r} 
   + |y|^2\log \frac{1}{r}   -|y|^2 \log \frac{e}{|y|} \right) (-\Delta^2 u(y))\,dy \\
  &+ \int_{r<|y|<1} \left(r^2 \log \frac{1}{|y|}\right) (-\Delta^2 u(y))\,dy \Big| \\
  \leq& r^2 \log \frac{1}{r} \int_{r<|y|<(\log \frac{1}{r})^{-1}}
  -\Delta^2 u(y)\,dy+r^2 \log \log \frac{1}{r} 
  \int_{|y|<1} 
-\Delta^2 u(y) \,dy \\
  &+r^2 \log \frac{e}{r} \int_{|y|<r} \Big| 1+\frac{e^2 \left(\frac{|y|}{er} \right)^2 \log \frac{|y|}{er}}{\log \frac{e}{r}} \Big| (-\Delta^2 u(y))\,dy \\
  =&o \left(r^2 \log \frac{e}{r} \right) \quad\text{as}\quad r\to 0^{+}.
 \end{align*}
 Thus averaging \eqref{eq7.12} and noting \eqref{eq7.19} we get
 \[
\bar{u}(r)=br^2 \log \frac{e}{r}+o(r^2 \log \frac{e}{r})
\quad\text{as}\quad r\to 0^+
\] which
 together with the positivity of $u$ implies
\begin{equation}\label{eq7.20}
 b\geq0.
\end{equation}

It follows from the integral form of the remainder in Taylor's theorem 
that if $x,y \in B_1 (0) \setminus \{0\}$ are such that $tx-y \neq0$ 
for all $t \in [0,1]$ then 
\[
\Phi(x-y)-\Phi(-y)-D\Phi(-y)x=-2A|x|^2 \int^{1}_{0} (1-t) \log
\frac{e}{|tx-y|} \,dt+\hat{\Phi}(x, y)
\]  
where
\[
  |\hat{\Phi}(x,y)| = A\Big| \frac{|x|^2}{2} +2\int^{1}_{0} (1-t) \frac{[(tx-y) \cdot x]^2}{|tx-y|^2} \,dt \Big| 
  \leq \frac{3}{2}A|x|^2 .
 \]
Thus for $x\in B_1 (0) \setminus \{0\}$ we have
\begin{equation}\label{eq7.21}
 U (x)=2A|x|^2 \int_{|y|<1} \int^{1}_{0} (1-t) 
\log \frac{e}{|tx-y|} \, dt 
(-\Delta^2 u(y)) \,dy+O(|x|^2 )
\end{equation}
by \eqref{eq7.7}.

By Fatou's lemma,
\begin{equation}\label{eq7.22}
 \liminf_{x\to 0} \int_{|y|<1} \int^{1}_{0} (1-t) 
\log \frac{e}{|tx-y|} \, dt 
(-\Delta^2 u(y))\,dy \geq \frac{1}{2} 
\int_{|y|<1} \log \frac{e}{|y|}\,  
 (-\Delta^2 u(y))\,dy.
\end{equation}

\noindent \emph{Case I.} Suppose $\int_{|y|<1} \left(\log \frac{e}{|y|} \right)
(-\Delta^2 u(y))\,dy=\infty$.  Then it follows from
\eqref{eq7.21}, \eqref{eq7.22}, \eqref{eq7.20}, \eqref{eq7.19}, and
\eqref{eq7.12} that $u(x)>>|x|^2$ as $x\to0$.  Thus, reversing the
original change of variables \eqref{eq7.1}, we have $v(y)>1$
for $|y|\geq r_0 /2$ for some $r_0 >2$ and $v(y)$ is a
solution of
\begin{equation}\label{eq7.23}
 -\Delta^2 v \geq |y|^{-\sigma}f(v)
\geq C|y|^{-\sigma}g(v)\geq C|y|^{-\sigma}v^{-1+\sigma/2}
\quad\text{ in }\mathbb{R}^2 \setminus B_{r_0 /2} (0)
\end{equation}
where $C$ is a positive constant and $g$ is the function in
\eqref{eq7.5}.  Averaging \eqref{eq7.23} we see that $\bar{v}(r)$ is a
positive radial solution of $-\Delta^2 \bar{v}\geq
C|y|^{-\sigma}\bar v^{-1+\sigma/2}$ in
$\mathbb{R}^2 \setminus B_{r_0 /2}(0)$ which contradicts Lemma \ref{lem2.8} and
completes the proof of Theorem \ref{thm1.3} in Case I.  \medskip

\noindent \emph{Case II.} Suppose $\int_{|y|<1} \left(\log
  \frac{e}{|y|} \right) (-\Delta^2 u(y))\,dy<\infty$.  Then
since 
\[
\log \frac{e}{|tx-y|}=\log \frac{e}{|y|}+\log
\frac{|y|}{|y-tx|}
\] 
we see that \eqref{eq7.21} and Lemma \ref{lem2.9} imply
$U (x)=O(|x|^2 )$ as $x\to 0$.  Hence if $b>0$ (resp. $b=0$)
then it follows from \eqref{eq7.12} and \eqref{eq7.19} that
\[
u(x)>>|x|^2 \quad (\text{resp. } u(x)=O(|x|^2 )) \quad\text{as}\quad x\to 0.
\]
If $u(x)>>|x|^2$
as $x\to0$ then we obtain a contradiction as in Case I.  Thus we can
assume for some $s_0>0$ that $|x|^{-2} u(x)<s_0$ for $0<|x|\le 1$.  Hence
reversing the original change of variables \eqref{eq7.1} we get
\begin{equation}\label{eq7.24}
 0<v(y)<s_0 \quad\text{for}\quad|y|\geq 1
\end{equation}
and $v(y)$ is a solution of
\begin{equation}\label{eq7.25}
 -\Delta^2 v\geq |y|^{-\sigma}f(v) \quad
\text{in}\quad \mathbb{R}^2 \setminus B_1 (0).
\end{equation}
We can assume $f|_{(0, s_0]}$ is $C^2$ and 
$(f|_{(0,s_0]})''>0$ because, as one easily verifies, the function
$\hat{f}:(0, s_0 ]\to(0,\infty)$ defined by 
\[
\hat{f}(s)=\frac{1}{s_0^2}
\int^{s}_{0} (s-\zeta) \left(\min_{\zeta \leq \tau \leq s_0 } f(\tau)
\right) d\zeta
\] 
is $C^2$ and satisfies 
\[0<\hat{f}(s)\leq f(s) \quad\text{and}\quad
\hat{f}^{''}(s)>0 \quad \text{for } 0<s\leq s_0.  
\]
Thus averaging \eqref{eq7.25} and noting \eqref{eq7.24} we find that
$\bar{v}(r)$ is a positive solution of $-\Delta^2 \bar{v}\geq
|y|^{-\sigma}f(\bar{v})$ in $\mathbb{R}^2 \setminus B_1 (0)$ which contradicts
Lemma \ref{lem2.8} and completes the proof of Theorem \ref{thm1.3} in
all cases.
\end{proof}

\appendix
\section{Represention formula and pointwise bound}\label{secA}

Let $\Phi$ be the fundamental solution of $\Delta^2$ in $\mathbb{R}^n$
given by \eqref{eq1.10}--\eqref{eq1.13} and
  for $x\neq0$ and $y\neq x$, let
\begin{equation}\label{eq3.5}
 \Psi(x,y)=\Phi(x-y)-\sum_{|\beta|\leq 1} \frac{(-y)^\beta}{\beta!} 
D^\beta \Phi(x)
\end{equation}
be the error in approximating $\Phi(x-y)$ with the partial sum of
degree one of the Taylor series of $\Phi$ at $x$.  

The following theorem, which we proved in \cite{GMT}, gives
representation formula \eqref{eq3.8} and pointwise bound \eqref{eq3.6}
for nonnegative solutions of
\begin{equation}\label{eqA.1}
-\Delta^2u\ge 0 \quad{in}\quad B_2(0)\setminus\{0\}\subset\mathbb R^n.
\end{equation}
See \cite{FKM} and \cite{FM} for some similar results.
\begin{thm}\label{thmA}
Let $u(x)$ be a $C^4$ nonnegative solution of \eqref{eqA.1} where
$n\ge 2$. Then 
\begin{equation}\label{eq3.6}
 u(x)=
\begin{cases}
O(|x|^{2-n})&\text{if $n\ge 3$}\\
O\left(\log\frac{e}{|x|}\right)&\text{if $n=2$}
\end{cases}
\quad \text{as}\quad x\to 0,
\end{equation}
\begin{equation}\label{eq3.7}
 \int_{|x|<1} |x|^2 (-\Delta^2 u(x))\,dx<\infty,
\end{equation}
and
\begin{equation}\label{eq3.8}
 u=N+h+\sum_{|\beta|\leq 2} a_\beta D^\beta \Phi \quad\text{in}
\quad B_1(0) \setminus \{0\},
\end{equation}
where $a_\beta$, $|\beta|\leq 2$, are constants, $h\in C^\infty
(\overline{B_1(0)})$ is a solution of $\Delta^2 h=0$ in
$\overline{B_1(0)}$, and
\begin{equation}\label{eq3.9}
 N(x)=\int_{|y|<1} \Psi(x,y)\Delta^2 u(y)\,dy \quad\text{for}\quad x\neq 0.
\end{equation}
\end{thm}

\n {\bf Acknowledgement}. We would like to thank Guido Sweers for
helpful comments.

\end{document}